\documentclass[a4paper, 12pt]{article}
\usepackage[T1]{fontenc}
\usepackage[utf8]{inputenc}
\usepackage[english]{babel}

\usepackage{geometry}

\usepackage[backend=biber,style=alphabetic]{biblatex}

\usepackage[style=german]{csquotes}

\usepackage{enumitem}
\setenumerate[1]{label=(\alph*)}

\usepackage{amsmath}
\usepackage{amssymb}
\usepackage{textcomp}
\usepackage{stmaryrd}

\usepackage[dvipsnames]{xcolor}

\usepackage{tikz}
\usetikzlibrary{cd,arrows,automata,shapes,calc,through,backgrounds,matrix,decorations.pathmorphing,positioning}

\usepackage{booktabs}

\newcommand{\notion}[1]{\emph{#1}\index{#1}}

\usepackage[thmmarks]{ntheorem}

\usepackage{colonequals}

\usepackage{relsize}

\usepackage{makeidx}

\newcommand{\inT}{\colon\hspace{-1.16mm}}

\newcommand{\ZZ}{\mathbb Z}
\newcommand{\RR}{\mathbb R}
\newcommand{\CC}{\mathbb C}

\newcommand{\id}{\mathrm{id}}

\newcommand{\Alg}{\ensuremath{\text{-Alg}}}

\newcommand{\op}{\ensuremath{\mathrm{op}}}
\newcommand{\bu}{\ensuremath{\mathsmaller{\mathsmaller{\,\bullet\,}}}}
\newcommand{\defequal}{\ensuremath{\colon\!\!\!\equiv}}

\newcommand{\image}{\ensuremath{\text{image}}}

\newcommand{\epi}{\ensuremath{\twoheadrightarrow}}
\newcommand{\mono}{\ensuremath{\hookrightarrow}}

\newcommand{\Aut}{\ensuremath{\mathrm{Aut}}}
\newcommand{\BAut}{\ensuremath{\mathrm{BAut}}}

\newcommand{\BG}{\ensuremath{\mathrm{BG}}}

\newcommand{\GL}{\ensuremath{\mathrm{GL}}}

\newcommand{\DD}{\ensuremath{\mathbb{D}}}

\newcommand{\refl}{\ensuremath{\mathrm{refl}}}

\newcommand{\Gstr}{\ensuremath{G\text{-str}}}
\newcommand{\hquot}{\hspace{-0.127em}\sslash\hspace{-0.127em}}
\newcommand{\agda}[1]{\href{#1}{\color{darkgreen} \ensuremath{✓}}}

\newcommand{\drawpb}[1]{\arrow[#1, phantom, "\text{\scriptsize(pb)}" description]}

\theorembodyfont{}
\theoremstyle{break}
    \newtheorem{theorem}{Theorem}[section]
    \newtheorem{corollary}[theorem]{Corollary}
    \newtheorem{remark}[theorem]{Remark}
    \newtheorem{lemma}[theorem]{Lemma}
    
    \newtheorem{proposition}[theorem]{Proposition}
    \newtheorem{definition}[theorem]{Definition}
    \newtheorem{axiom}[theorem]{Axiom}
    
    \newtheorem{example}[theorem]{Example}
    \newtheorem{examples}[theorem]{Examples}
\theoremstyle{plain}

\theoremstyle{nonumberbreak}
\theoremstyle{nonumberplain}
\theoremsymbol{\ensuremath{\qed}}
    \newtheorem{proof}{Proof}
    \newtheorem{construction}{Construction}

\newcount\colveccount
\newcommand*\colvec[1]{
        \global\colveccount#1
        \begin{pmatrix}
        \colvecnext
}
\def\colvecnext#1{
        #1
        \global\advance\colveccount-1
        \ifnum\colveccount>0
                \\
                \expandafter\colvecnext
        \else
                \end{pmatrix}
        \fi
}

\newcommand{\ignore}[1]{}

\usepackage[colorlinks=true, linkcolor=Black, citecolor=Black]{hyperref}
\RequirePackage[hyphenbreaks]{breakurl}
\definecolor{darkgreen}{rgb}{0,0.5,0}
\RequirePackage{cleveref}

\makeindex

\begin{document}

\begin{center}
  \LARGE{Synthetic G-Jet-Structures in \\ Modal Homotopy Type Theory \\}
  \vspace{0.7cm}
  \normalsize{Felix Cherubini \\}
  \vspace{0.2cm}
  \normalsize{University of Gothenburg and Chalmers University of Technology \\}
  \vspace{0.2cm}
  \normalsize{\today \\}
  \vspace{0.2cm}
  \begin{abstract}
    This article constructs the moduli stack of torsion-free $G$-jet-structures in homotopy type theory with one monadic modality.
    This yields a construction of this moduli stack for any $\infty$-topos equipped with any stable factorization systems.

    In the intended applications of this theory, the factorization systems are given by the deRham-Stack construction.
    Homotopy type theory allows a formulation of this abstract theory with surprising low complexity.
    This is witnessed by the accompanying formalization of large parts of this work.
  \end{abstract}
\end{center}

\tableofcontents

\section{Introduction}

The constructions and theorems in this article are formulated in homotopy type theory.
In \cite{shulman-strict-univalence}, Michael Shulman has shown,
that homotopy type theory can be interpreted in any Grothendieck $(\infty,1)$-topos as defined in \cite{Lurie}[Definition 6.1.0.4].
Throughout the article, we assume a fixed monadic modality.
By \notion{(monadic) modality} we mean the same as ``modality'' defined in \cite{UFP}[Definition 7.7.5] or the ``higher modalities'' \cite{egbert-bas-mike}[Definition 1.1] or the equivalent notion of ``uniquely eliminating modalities'' \cite{egbert-bas-mike}[Definition 1.2].

A modality may be described as an operation $\Im$ together with a map $\iota_X:X\to \Im X$ for any type $X$,
such that a dependent version of the following commonly known property of a reflector holds:

For all $Y$ such that $\iota_Y:Y\to \Im Y$ is an equivalence and all
maps $f:X\to Y$, there is a unique $\psi:\Im X\to Y$, such that the diagram commutes
\begin{center}
  \begin{tikzcd}
    X\arrow[r,"\iota_X"]\arrow[dr, "f", swap] & \Im X\arrow[dashed, d,"\exists!\psi"] \\
    & Y
  \end{tikzcd}
\end{center}
The dependent version of this universal property will be axiom \ref{axiom-coreduction} -- which we assume throughout this article for convenience.
Externally, a monadic modality in homotopy type theory corresponds to a stable factorization system on an $(\infty,1)$-topos \cite{egbert-bas-mike}[Appendix A, in particular p.\ 76].

The examples of modalities $(\Im,\iota)$ we had in mind when writing this article should be thought of as providing a notion of \emph{infinitesimally close}. More specifically, two points $x,y:X$ are infinitesimally close if their images under $\iota$ coincide.
This makes only sense in a context where there are infinitesimals in the first place.

As far as the author knows, all relevant examples of such particular modalities are constructed by passing from spaces to algebras of functions on spaces and by introducing infinitesimals via nilpotent elements in those algebras.
A good intuition is that the functions are coordinate functions and in an infinitesimal space the coordinates can be so small that taking a power of them actually turns them into zero. If these infinitesimal spaces are around, macroscopic spaces $X$ can be probed by them.

The information, that can be probed in this way, may be collapsed by passing to $\Im X$. It is important to note that this collapse almost never preserves structured spaces like manifolds or schemes -- they are replaced by macroscopically similar spaces, which have trivial infinitesimal structure and therefore trivial tangent spaces.
Spaces which are only of infinitesimal extent, like the formal or $k$-th order disks of algebraic geometry, are mapped to the one point space by $\Im$.
We will sketch an easy way of constructing $\Im$ below, which works for a class of examples.
It turns out to be more natural to have a modality $\Im$ which collapses infinitesimals of all orders at once --
it is possible to construct models which capture the notion of a similar modality collapsing only first order infinitesimals, but the models the author came up with lacked other desirable properties.
The use of $G$-jet-structures in this article instead of $G$-structures stems from this decisions for, to use general infinitesimals as a primary notion.
In \Cref{G-jet-structures} we will briefly explain why we do expect that we can also cover the case of $G$-structures with $G$-jet-structures in the case of $G=\mathrm{GL}(n,\mathbb R)$.

Urs Schreiber and Igor Khavkine define basic notions of differential geometry as well as generalized partial differential equations in \cite{SyntheticPDEs} and most of their constructions, as they note, do not depend on the particular topos and the particular modality ``$\Im$'' they use. Crucially, they show that in the topos they use, the abstract definitions of formal disks and formally étale maps, analogous to \Cref{def:formal-disk} and \Cref{formally-etale} in this article, coincide with formal disks in manifolds and local diffeomorphisms of manifolds.

In the appendix of \cite{cherubini_rijke_2021}, it is shown, that in the Zariski topos,
the definitions of formal disks and formally étale maps \Cref{formally-etale}, \Cref{def:formal-disk} correspond to usual formal neighborhoods and formally étale maps of algebraic geometry. It is certainly noteworthy, that the abstract theory in this article combines quite well with synthetic differential geometry, which is used extensively in the preprint \cite{david-orbifolds}.

In \cite{SchreiberHamburg}, Urs Schreiber presented a couple of problems together with proposals for their solution to the homotopy type theory community. This article solves one of these problems, which is the construction of the moduli space of torsion-free G-jet-structures \Cref{def:space-of-torsion-free-jet-structures}, where \Cref{triviality-theorem} is an important step also mentioned by Schreiber. The proof of the latter theorem in this article is a vast simplification of Schreiber's proof, which relied on pasting of homotopy pullbacks, where the proof in this article uses simpler reasoning with dependent types. A solution to Schreiber's problem was already given in the phd thesis of the author \cite{wellen-thesis}, but not published under peer review. 

A minor difference to the construction of the moduli space proposed in \cite{SchreiberHamburg}
is that the G-jet-structures are checked for triviality on first-order infinitesimal disks,
while for this article, after discussing with Schreiber, full formal disks are used everywhere.
It is left to check in future work, that the construction given here type-theoretically,
yields the same space as the classical construction.

Important advantages of homotopy type theory for this work include the unusual conciseness for a higher categorical framework.
Furthermore, a proof-assistant software, in this case \notion{Agda}, can be used to check definitions and proofs written out in homotopy type theory.
This was of great help to the author during the development of the theory in this article and while learning the subject. 
The formalization can be viewed at \href{https://github.com/felixwellen/DCHoTT-Agda}{https://github.com/felixwellen/DCHoTT-Agda}, where \Cref{triviality-theorem} is to be found in the file \href{https://github.com/felixwellen/DCHoTT-Agda/blob/master/FormalDiskBundle.agda}{FormalDiskBundle.agda} and the central construction \Cref{def:space-of-torsion-free-jet-structures} in \href{https://github.com/felixwellen/DCHoTT-Agda/blob/master/G-structure.agda}{G-structures.agda}. \\

We will conclude this introduction by giving more intuition for the intended models.
This part is aimed in par\-ti\-cu\-lar at readers not familiar with higher stacks or synthetic differential geometry.

An important thing to note is that manifolds and other simple spaces of interest in differential geometry, are, maybe to the surprise of some readers, \emph{not} to be thought of
as \emph{higher} types. Note that this is also the case for the topological spaces in \cite{ShulmanRealCohesion}.
Instead, in the applications of interest, a manifold is usually a 0-truncated type.
The higher types in this context are given by passing from the ordinary, 1-categorical notion of Grothendieck toposes, to their higher categorical version. The latter includes the former, as the subcategory of 0-truncated objects. Thus, the spaces of interest, which already exist in the 1-categorical topos, are included in the 0-truncated types.

The theory in this article may however also be applied to objects more general than manifolds,
which are not 0-truncated.
One important example are quotient stacks. In addition, it is also possible to consider spaces,
which are not locally modeled on 0-truncated types. Both cases are not ruled out by \Cref{def:V-manifold}.

Furthermore, the ambient higher types admit the construction of classifying morphisms (see \Cref{definition-classifying-square-for-fiber-bundle}) of fiber bundles, which is crucial for the goals of the article. In addition to that, there are exceptionally easy ways to describe homotopy theoretic quotients of spaces by simple type theoretic constructions, which rely on higher identity types as well. This will be explained in the discussion preceding \Cref{def:space-of-torsion-free-jet-structures}.

The kind of modality that can be used to access the differential geometric structure of the objects of a topos from within type theory,
is in some fortunate cases generated by reducing algebras.
More precisely, in one of the most basic models,
namely simplicial sheaves on the category of $k\Alg^\op_{fp}$ finitely presented algebras over a field $k$
\footnote{This ensures that the nilradical is finitely generated -- if this is not the case, the definition of $\Im$ becomes more complicated.},
there is an endofunctor $\Im$ given by
\[ (\Im X)(A):\equiv X(A/\sqrt{0})\]
for any sheaf $X$. If reduction preserves covers, as it does for the Zariski topology, this is an idempotent left and right adjoint functor, which is enough to generate a modality on the topos.
The same approach yields modalities on toposes suitable for differential geometry. 
Roughly, this is achieved by passing to algebras of smooth functions and taking tensor products with nilpotent algebras, to add infinitesimals to the theory (see \cite{SyntheticPDEs}).

It is also possible to only add square-zero algebras instead of general nilpotent algebras, which makes the definition of formal disks (\Cref{def:formal-disk}) collapse to first-order neighborhoods - something very close to a tangent space. This leads to a simpler theory, which is easier to compare with differential geometry, but it also yields a category which doesn't have the right limits, so we will not consider it any further.

The functor $\Im$ appears in the differential part of a \notion{Differential Cohesive Topos}, a notion due to Urs Schreiber \cite{SchreiberDcct}[Definition 4.2.1], extending Lawvere's Axiomatic Cohesion \cite{Lawvere07}.
The differential structure is also used on toposes of Set-valued sheaves \cite{SyntheticPDEs},
where it is applied to a site suitable for differential geometry and therefore spaces modeled on vector spaces over the reals.

Since this modality $\Im$, that we will use in our type theory, allows us to build at least some abstract differential geometry relative to it,
one might ask what role the external functor $\Im$ from above plays in conventional geometry.
The answer is that concepts very close to it appear very early in the Grothendieck school of algebraic geometry,
which is no surprise at all, since algebras with nilpotent elements were specifically used to admit reasoning with this kind of infinitesimals.
However, the functor itself leaves the impression of a rather exotic concept under the names of deRham prestack \cite{GaitsgoryCrystals}, deRham stack, deRham space or infinitesimal shape and is usually used to represent D-modules over a smooth scheme or algebraic stack $X$ as quasi-coherent sheaves over $\Im X$.
A functor $\Im$ also exists in meaningful ways in non-commutative geometry \cite{KRSpacesNice}.
In the face of these rather advanced use cases of $\Im$, it might be irritating that we use it as a basis for differential geometry.
One reason $\Im X$ appears so infrequent in geometry might be, that it is quite hard to build intuition for what it is like as a space.
If $X$ is a structured space like a manifold or a scheme, $\Im X$ will only be a manifold or scheme in degenerate cases.
On the other hand, the relation provided by the map $\iota_X:X\to\Im X$ can be understood quite intuitively as ``\notion{infinitesimally close}''.
This is how we will start to develop differential geometry based on $\Im$. \\

\subsection*{Content}
\begin{itemize}
\item We define the formal disk at a point in a type in \ref{def:formal-disk}.
  These disks contain roughly similar information as the tangent and jet spaces in differential geometry.
  The definition is relative to a modality and for the $n$-truncation modality known as the connected cover of a homotopy type.
\item We introduce a notion of homogeneous type in \ref{def-homogeneous-type},
  which is tailored to our application as a basic building block for manifolds.
  It is proven, that the formal disk bundle of a homogeneous type is trivial.
\item Formally étale maps are defined in \ref{formally-etale}.
  Between manifolds, formally étale maps are known to correspond to local diffeomorphisms.
  We show stability properties of the class of formally étale maps, for example closure under arbitrary pullbacks.
  This definition is again relative to a modality.
\item Multiple definitions of fiber bundle are shown to be equivalent in \ref{thm-equivalence-of-fiber-bundle-definitions}.
  Notably, we show that if all fibers of a map are merely equal to a fixed type,
  then there is a trivializing cover.
\item For homogeneous types $V$, we define $V$-manifolds in \ref{def:V-manifold}.
  They are spaces infinitesimally modeled on $V$.
\item Finally, we define $G$-jet-structures in \ref{def:G-jet-structure} and their moduli space for a given ma\-ni\-fold.
  We also define torsion-free $G$-jet-structures and show that the trivial 1-jet-structure of a 1-group is torsion-free.
\end{itemize}
This project was suggested by Urs Schreiber in 2015 as a PhD thesis project for the author.
The (external) definitions of formally étale maps, $V$-manifolds and $G$-jet-structures have been used by Urs Schreiber and others.
Our contribution is the formulation in homotopy type theory and type-theoretic solution of the proposed problems
\footnote{These were the triviality of the formal disk bundle on a homogeneous type,
  local triviality of the formal disk bundle of a $V$-manifolds and definition of $G$-structures and torsion-free $G$-structures.},
which allowed us to produce a theory of low complexity and high clarity,
which is hard to imagine to be possible in a more classical framework like higher category theory in its simplicial incarnation.

\subsection*{Formalization}
The formalization located here:
\begin{center}
  \href{https://github.com/felixwellen/DCHoTT-Agda}{https://github.com/felixwellen/DCHoTT-Agda}
\end{center}
covers everything up to and including the definition of $G$-jet-structures, but not definitions building on top of that. However, crucial ingredients for the construction of the moduli-space of $G$-jet-structures and torsion-free $G$-jet-structures, like the chain rule, are checked.
It turned out that the necessary engineering work to actually combine those ingredients is not justified by the gain in understanding. Furthermore, before the code is used as a basis for future work, it should be ported to a suitable library. 

\subsection*{Acknowledgments}
The idea of using modalities in homotopy type theory in the way present in this work is due to Urs Schreiber and Mike Shulman \cite{SchreiberHamburg} \cite{ShulmanSchreiber},
Schreiber was one of the supervisors of the author's thesis.
He provided the author with all the categorical versions of the important geometric definitions,
as well as the main theorems and category theoretic proofs leading to the type theoretic version of his Higher Cartan Geometry
presented in this article.
Adaptions to homotopy type theory of Schreiber's original proofs are included in the author's thesis \cite{wellen-thesis}.
The proofs in this article make more use of type theoretic dependency which shortens the arguments a lot in most cases.
Some concepts needed reformulation and additional theorems were needed to make the main result,
the construction of moduli spaces of torsion-free $G$-jet-structures, possible.

Schreiber explained a lot of mathematics important to this work to the author on his many visits in Bonn
and answered countless questions via email.

During the time of writing his thesis and on later occasional visits,
the author profited a lot from his working groups in Karlsruhe.
This work wouldn't be the same without the discussions with and the Algebra knowledge of Tobias Columbus and Fabian Januszewski and the support of Frank Herrlich, Stefan Kühnlein and other members of the Algebra group and the Didactics group.
On a couple of visits in Darmstadt, Ulrik Buchholtz, Thomas Streicher and Jonathan Weinberger listened carefully to various versions of the theory in this article and made lots of helpful comments. Two questions of Ulrik Buchholtz led directly to propositions in this article (part of \ref{thm-equivalence-of-fiber-bundle-definitions} and \ref{generate-manifolds}).

A short visit in Nottingham and discussion with Paolo Capriotti, Nicolai Kraus and Thorsten Altenkirch also helped in the early stages of the theory and had an impact on the authors agda knowledge.

The discussion with the Mathematics Research Community group, helped the author a lot to understand Differential Cohesion better.
The research events in this line were sponsored by the National Science Foundation under Grant Number DMS 1641020.
The group work for the Differential Cohesion group at this event was organized by Dan Licata and Mike Shulman.
The group member Max S.\ New later read part of the thesis and made an important suggestion for an improvement of the definition of fiber bundle.

The improvements on this work were developed on a Postdoc position in Steve Awodey's group at Carnegie Mellon University,
sponsored by The United States Air Force Research Laboratory
under agreement number FA9550-15-1-0053.
The good atmosphere with lots of opportunities of discussion with local homotopy type theorists as well as the many visitors helped a lot.
Steve Awodey gave the author lot's of opportunities to present his work and new ideas to the locals and the guests and made lots of helpful discussions possible.
One consequence important to this work was a joint, successful effort with Egbert Rijke, to understand formally étale maps better ---
another countless discussions with Jonas Frey about abstract Geometry, the role of higher categorical structures therein and Type and Category Theory in general.
During that time in Pittsburgh, comments of and discussions with Mike Shulman, Mathieu Anel, André Joyal, Eric Finster, Dan Christensen and Marcelo Fiore led to improvements and helped the author to understand many things important to this article better.

Finally, the author is very thankful to two anonymous reviewers who read the article with great care and had many helpful comments.
The questions of the reviewers led to improvements of the content and their suggestions improved the presentation a lot.
Work on these revisions was carried out on a position in Gothenburg which was funded by US-Army grant W911NF-21-1-0318 and the ForCUTT project, ERC advanced grant 101053291.

\subsection*{Competing interests}
There are no competing interests.\footnote{This remark was added, because the submission process of MCSC requires it.}

\section{Modal homotopy type theory}

\subsection{Terminology and notation}

Mostly, we use the same terminology and notation as the HoTT-Book \cite{UFP}. 
However, there are a few exceptions. 
To denote terms of type $\prod_{x\inT A}B(x)$ we use the notation for $\lambda$-expressions from pure mathematics, i.e. $x\mapsto f(x)$.
There are no implicit propositional truncations. 
If the propositional truncation of a statement is used, it is indicated by the word ``merely''.
Phrases like ``for all'' and ``there is'' are to be interpreted as $\prod$- and $\sum$-types.
For example, the sentence
\begin{center}
  For all $x\inT A$ we have $t\inT B(x)$.
\end{center}
is to be read as the statement describing the term $(x\inT A)\mapsto t$ of type $ \prod_{x\inT A} B(x)$.
We sometimes write $f_a$ for the application of a dependent function $f\inT \prod_{x\inT A} B(x)$ to $a\inT A$, instead of $f(a)$.

Furthermore, similar to \cite{ShulmanRealCohesion}, when dealing with identity types, we avoid topology and geometry related words. 
For example, we write ``equality'' instead of ``path'' and ``2-cell'' instead of ``homotopy'', 
to avoid confusion with the notions of paths and homotopies for the classical geometric objects we like to study by including them in our theory as 0-types.
We use $p \bu q$ to denote the concatenation of equalities $p$ and $q$.
We say that $x$ is unique with some properties, if the type of all $x$ with these properties is contractible.

\subsection{Preliminaries from homotopy type theory}

We use a fragment of the Type Theory from \cite{UFP}.
Function extensionality is always assumed to hold.
Furthermore, we assume a propositional truncation modality ``$\|\_\|$'' and univalent universes.

In the next section we will give axioms for a modality ``$\Im$'', which will be assumed throughout the article.
Some knowledge of the basic concepts in \cite{UFP} is assumed.
In addition, we will use more facts about pullbacks than presented in \cite{UFP}, which we will list in this section.

It is very useful to switch between pullback squares and equivalences \emph{over} a morphism.
We start with the latter concept.
\begin{definition}
  Let $f:A\to B$ be a map and $P:A\to \mathcal U$, $Q:B\to \mathcal U$ be dependent types.
  \begin{enumerate}
  \item A \notion{morphism over $f$} or \notion{fibered morphism} is a
    \[ \varphi:\prod_{x:A}P(x) \to Q(f(x)). \]
  \item An \notion{equivalence over $f$} or \notion{fibered equivalence} is a
    \[ \varphi:\prod_{x:A}P(x) \simeq Q(f(x)). \]
  \end{enumerate}
\end{definition}
For every morphism over $f:A\to B$ as above, we can construct a square
\footnote{By stating that it is a ``square'' we implicitly assume that there is a 2-cell letting it commute, which is considered to be part of the square.
  In this particular case, the 2-cell is trivial.}
\begin{center}
  \begin{tikzcd}[auto]
    \sum_{x:A} P(x)\arrow[r]\arrow[d,"\pi_1",swap] & \sum_{x:B} Q(x)\arrow[d,"\pi_1"] \\
    A\arrow[r,"f",swap] & B \\
  \end{tikzcd}
\end{center}
where the top map is given as $(a,p_a)\mapsto (f(a),\varphi_a(p_a))$. 
This square will turn out to be a pullback in the sense we are going to describe now, 
if and only if $\varphi$ is an equivalence over $f$.

For a cospan given by the maps $f:A\to C$ and $g:B\to C$,
we can construct a pullback square:
\begin{center}
  \begin{tikzcd}[auto]
    \sum_{x:A, y:B}f(x)=g(y) \arrow[r,"\pi_2",swap]\arrow[d,"\pi_1"] & B\arrow[d,"g"] \\
    A\arrow[r,"f",swap] & C \\
  \end{tikzcd}
\end{center}
Then, for any other completion of the cospan to a square
\begin{center}
  \begin{tikzcd}[auto]
    X \arrow[d,swap,"\varphi_A"] \arrow[r,"\varphi_B"] & |[alias=B]| B \arrow[d,"g"] \arrow[Rightarrow, from=A, to=B,shorten <= 1em, shorten >= 1em,"\eta"] \\
    |[alias=A]| A \arrow[r,"f",swap] & C
  \end{tikzcd}
\end{center}
where $\eta:\prod_{x:X}g(x)=f(x)$ is a 2-cell letting it commute,
an induced map to the pullback is given by $x\mapsto (\varphi_A(x),\varphi_B(x),\eta_x)$.
\begin{definition}
  A square is given by four maps as above and a 2-cell like $\eta$. A square is a \notion{pullback square} if the induced map described above is an equivalence.
\end{definition}
To reverse the construction of a square for a morphism over ``$f$'' above,
we can start with a general square:
\begin{center}
  \begin{tikzcd}[auto]
    X \arrow[d,"p_A",swap] \arrow[r,"g"] & |[alias=Y]| Y \arrow[d,"p_B"] \arrow[Rightarrow, from=Y, to=A,shorten <= 1em, shorten >= 1em,"\eta"] \\
    |[alias=A]| A \arrow[r,"f",swap] & B
  \end{tikzcd}
\end{center}
Let $P:A\to\mathcal U$ and $Q:B\to\mathcal U$ be the fiber types of the vertical maps, 
i.e.
\begin{align*}
  &P(a:A):\equiv \sum_{x:X}p_A(x)=a \\
  &Q(b:B):\equiv \sum_{y:Y}p_B(y)=b 
\end{align*}
Then, for all $a:A$, a morphism $\varphi_a:P(a)\to Q(a)$ is given as 
\[ \varphi_a((a,(x,p))):\equiv (f(a),(g(x),\eta_x\bu f(p))).\]
So $\varphi$ is a morphism from $P$ to $Q$ over $f$.
The following statement is quite useful and will be used frequently in this article:
\begin{lemma}
  \label{rem-pullback-fiberwise-equiv}
  \begin{enumerate}
  \item A square is a pullback if and only if the induced fibered morphism is an equivalence.
  \item A fibered morphism is an equivalence, if and only if the corresponding square is a pullback.
  \end{enumerate}
\end{lemma}
Now, the following corollary can be derived by using the fact that equivalences are stable under pullback:
\begin{corollary}
  \label{sum-equivalence-over}
  Let $f:A\to B$ be an equivalence, $P:A\to\mathcal U$, $Q:B\to\mathcal U$ dependent types and $\varphi:\prod_{x:A}P(x)\to Q(f(x))$ an equivalence over $f$.
  Then the induced map 
  \[ \left(\sum_{x:A}P(x)\right) \to \left(\sum_{x:B} Q(x)\right)\] 
  is an equivalence.
\end{corollary}

\subsection{Modalities}

From this section on, we will always assume a \emph{modality} $\Im$.
We use the definition of a uniquely eliminating modality from \cite{egbert-bas-mike},
which is equivalent to the definition given in \cite[Section 7.7]{UFP}.
More on modalities and their relation to concepts in category theory can be found in \cite{egbert-bas-mike}.
We deviate from the usual symbol for modalities, which would be ``$\bigcirc$'' to remind us that while we will technically work with a general modality,
we have some particular kind of modality in mind.
Furthermore, the work in this article could be reused in a type theory which provides more modal operators from differential cohesion,
for example in the work in progress \cite{david-orbifolds} which also uses homotopy type theory as a basis and where $\Im$ is called \emph{crystalline modality}.

The modality $\Im$ is also used in cateogry theory based differential cohesion\footnote{For example in \cite{SchreiberDcct} and \cite{SyntheticPDEs}.} and is called
infinitesimal shape\footnote{In the literature outside of differential cohesion, there is also the name ``deRham stack''.}.

\begin{axiom}
  \label{axiom-coreduction}
  From this point on, we assume existence of a map $\Im: \mathcal U\to \mathcal U$ 
  and maps $\iota_A: A \to \Im A$ for all types $A$, subject to this condition:
  For any $B:\Im A\to\mathcal U$, the map
  \[ \_\circ\iota_A:\left(\prod_{a\inT \Im A}\Im B(a)\right) \to\left(\prod_{a\inT A}\Im B(\iota_A(a))\right)  \]
  is an equivalence.
\end{axiom}
We call the inverse of the equivalence \notion{$\Im$-elimination}.
Elimination in type theory is a principle which lets us define maps starting in an inductive type like the natural numbers.
For example,
eliminating from the natural numbers $\mathbb{N}$ to a dependent proposition $P:\mathbb{N}\to \mathcal U$ means essentially to prove the proposition for each possible way to construct a natural number, which is either to take it to be the constant $0$ or the successor $s(n)$ of another natural number $n$.

The analogy to $\Im$-elimination is, that to eliminate from $\Im A$ into the dependent modal type $\Im B(\_)$,
we only need to provide a value for the case that $x:\Im A$ is of the form $\iota_A(y)$.
This is exactly what the inverse of the map in axiom \ref{axiom-coreduction} allows us to do.
A different way to put this is that $\Im A$ has the same elimination principle as a inductive type with constructor $\iota_A:A\to \Im A$
would have, except that it can only be used to construct functions with modal codomain.

Note that it is possible to conclude a variant of $\Im$-elimination from axiom \ref{axiom-coreduction},
where $\Im$ is not applied to the type family $B$, but the type family is required to have values in $\Im$-modal types\footnote{See \Cref{def-modal}.}.

Note that the equivalence in axiom \ref{axiom-coreduction} specializes to the universal property of a reflection if the family $B$ is constant:
\begin{center}
  \begin{tikzcd}[auto]
    A\arrow[r,"\iota_A"]\arrow[dr,swap,"f"] & \Im A\arrow[d,dashed,"\exists!\psi"] \\
     & \Im B 
  \end{tikzcd}
\end{center}
i.e. for all types $B$ and all $f:A\to \Im B$, we get a unique $\psi$ letting the triangle commute up to a 2-cell.
Unique means here, there is a contractible type of maps with 2-cells letting the triangle commute.
That type is also a fiber of the equivalence ``$\_\circ\iota_A$'', so we do know that it is contractible.

We will make use of this in showing that $\Im$ is idempotent in the following sense:
\begin{proposition}
  For all types $A$, the map $\iota_{\Im A}:\Im A\to \Im(\Im A)$ is an equivalence.
\end{proposition}
\begin{proof}
  By the universal property we just discussed, we get a candidate for an inverse to $\iota_{\Im A}$, which we call $\varphi$:
  \begin{center}
    \begin{tikzcd}[auto]
      \Im A\arrow[r,"\iota_{\Im A}"]\arrow[dr,swap,"\id"] & \Im (\Im A)\arrow[d,dashed,"\exists!\varphi"] \\
      & \Im A 
    \end{tikzcd}
  \end{center}
  By construction, $\varphi$ is already a left inverse of $\iota_{\Im A}$.
  We consider the diagram
  \begin{center}
    \begin{tikzcd}[auto]
      & \Im (\Im A)\arrow[d,"\varphi"]\arrow[dd, bend left=60, "\id"] \\
      \Im A\arrow[ru,"\iota_{\Im A}"]\arrow[r,swap,"\id"]\arrow[dr, swap, "\iota_{\Im A}"]  & \Im A\arrow[d, "\iota_{\Im A}"] \\
      & \Im (\Im A)
    \end{tikzcd}
  \end{center}
  and conclude that $\varphi$ is also a right inverse by uniqueness.
\end{proof}

Like reflections determine a subcategory, $\Im$ determines a subuniverse of the universe $\mathcal U$ of all types
\footnote{We implicitly assume a hierarchy of universes $\mathcal U_i$, but only mention indices if there is something interesting to say about them. }.
\begin{definition}
  \label{def-modal}
  \begin{enumerate}
  \item A type $A$ is \notion{$\Im$-modal} if $\iota_A$ is an equivalence.
  \item The \notion{universe of $\Im$-modal types} is
    \[ \mathcal U_\Im:\equiv\sum_{A:\mathcal U}(\text{$A$ is $\Im$-modal}) \]
  \end{enumerate}
\end{definition}
From what we proved above, all types $\Im A$ will be modal. 

As we explained in the introduction, we will not be very interested in spaces of the form ``$\Im X$'',
but more in the ``quotient map'' $\iota_X:X\to \Im X$, which we will view as identifying infinitesimally close points.

Like a functor, $\Im$ extends to maps and we get a naturality square for $\iota$:
\begin{definition}
  \label{definition-containing-naturality}
  \begin{enumerate}[label=(\roman*)]
  \item For any function $f: A\to B$ between arbitrary types $A$ and $B$, we have a function:
    \[ \Im f : \Im A \to \Im B  \]
    given by $\Im$-elimination.
  \item For any function $f: A\to B$ between arbitrary types $A$ and $B$, 
    there is a 2-cell $\eta$ witnessing that the following commutes:
    \begin{center}
      \begin{tikzcd}[auto]
        A\arrow[r,"\iota_A"]\arrow[d,swap,"f"] & \Im A\arrow[d,"\Im f"]\arrow[dl, Rightarrow, shorten <= 1em, shorten >= 1em,"\eta_f"] \\
        B\arrow[r,"\iota_B", swap] & \Im B 
      \end{tikzcd}
    \end{center}
  \end{enumerate}
\end{definition}
It is also straightforward to prove that the application of $\Im$ to maps commutes with composition of 
maps up to equality and preserves identities up to equality.
And in general, we expect that any coherence between these equalities needed in practice can be constructed.

\begin{remark}
   For any 2-cell $\eta: f \Rightarrow g$, we have a 2-cell between the images:
    \[ \Im \eta: \Im f \Rightarrow \Im g\text{.}\]
\end{remark}

$\Im$-Modal types have various closedness properties, which we review in the following lemma.
\begin{proposition}
  \label{lemma:retracts-pi-sigma}
  Let $A$ be any type and $B: A\to \mathcal U$ be such that for all $a\inT A$ the type $B(a)$ is $\Im$-modal.
  \begin{enumerate}
  \item Retracts of $\Im$-modal types are $\Im$-modal.
  \item The dependent product
    \[ \prod_{a\inT A}B(a)\]
    is $\Im$-modal. Note that $A$ is not required to be $\Im$-modal here and this implies all function spaces with $\Im$-modal codomain are $\Im$-modal.
  \item If $A$ is $\Im$-modal, the sum 
    \[ \sum_{a\inT A}B(a) \]
    is $\Im$-modal.
  \item $\Im$-modal types have $\Im$-modal identity types.
  \end{enumerate}
\end{proposition}
\begin{proof}
  \begin{enumerate}
  \item   A type $R$ is a retract of $B$ if there are maps $r\inT B \to R$ and $\iota\inT R \to B$, 
    such that $r\circ \iota$ is equal to the identity. 
    For all $\Im$-modal $B$ and retracts $R$ of $B$ we have the following diagram:
    \begin{center}
      \begin{tikzcd}
          R\arrow[d, "\iota_R"]\arrow[r, "\iota"]\arrow[rr, bend left, "\id"] & B\arrow[d, "\iota_B"]\arrow[r, "r"] & R\arrow[d, "\iota_R"] \\
          \Im R\arrow[r, "\Im\iota"]\arrow[rr, bend right, "\id"] & \Im B\arrow[r, "\Im r"] & \Im R \\
      \end{tikzcd}
    \end{center}
    Since $\iota_B$ is an equivalence, it has an inverse and by the diagram, 
    $r \circ \iota_B^{-1} \circ\Im\iota$ is a biinverse to $\iota_R$.
  \item This is proved, up to equivalence, in \cite[Theorem 7.7.7]{UFP}.
  \item This is \cite[Theorem 7.7.4]{UFP}.
  \item This is \cite[Lemma 1.25]{egbert-bas-mike}.
  \end{enumerate}
\end{proof}

One immediate consequence is $\Im 1 \simeq 1$ -- this is the only provably $\Im$-modal type.
We can not expect to prove more types to be $\Im$-modal, since there is always the modality
that maps all types to $1$, so $1$ \emph{could} be the only $\Im$-modal type.

The following is a slight variation of \cite{egbert-bas-mike}[Lemma 1.24], and plays a central role in the abstract \cite{oxford-abstract}, which was the beginning of \cite{cherubini_rijke_2021}:
\begin{proposition}
  \label{lemma-compute-modal-sum}
  Let $A$ be a type and $B\inT \Im A \to \mathcal U$ a dependent type. Then the induced map is an equivalence:
  \[ \Im \left(\sum_{x\inT A} B(\iota_A(x))\right) \simeq \left(\sum_{x\inT \Im A} \Im(B(x)) \right). \]
\end{proposition}

A more category theoretic implication of this proposition is that for the map
\[
        \pi_1: \left(\sum_{x\inT A} B(\iota_A(x))\right) \to A
\]
taking fibers commutes with application of $\Im$. Here, $\pi_1$ is an example of a formally étale map, which we will introduce in the next section.
More abstractly, this relates to the principle in algebraic topology,
that homotopy fibers coincide with ordinary fibers of certain fibrations.
This point is highlighted and used in \cite{david-fibrations}.

\section{A basis for differential geometry}
\subsection{Formal disks}
\label{section-formal-disks}

We will start to build geometric notions on top of the modality $\Im$ and its unit $\iota$.
In the intended applications the modality $\Im$ provides us with a notion of infinitesimal proximity.
To see if two points $x,y$ in some type $A$ are infinitesimally close to each other, 
we map them to $\Im A$ and ask if the images are equal.
\begin{definition}
  Let $x,y\colon A$. Then we have a type which could be read ``$x$ is \notion{infinitesimally close} to $y$'' and is given as:
  \[ x\sim y:\equiv (\iota_A(x)=\iota_A(y)). \]
\end{definition}
Of course, this is in general not a proposition, but it is useful to think about $\iota_A(x)=\iota_A(y)$ in this way.
The name ``infinitesimally close'' is a poor choice for a general modality\footnote{The concepts we will build up in this section are still of interest for other modalities, but this will be less and less true towards the end of this article.}, so the reader should keep in mind from now on, that the terminology is adapted to a modality in the intended applications\footnote{Which we describe in the introduction.}.

It turns out that all morphisms of types already respect this notion of infinitesimal closedness,
i.e.\ if two points are infinitesimally close to each other their images are close as well.

\begin{remark}
  \label{mapping-preserves-infinitesimal-proximity}
  If $x,y\colon A$ are infinitesimally close then for any map $f\colon A\to B$, the images $f(x)$ and $f(y)$ are infinitesimally close.
  More precisely, we have an induced function
  \[ \tilde{f}:(x\sim y)\to (f(x)\sim f(y))\]
\end{remark}
\begin{proof}
  We construct a map between the two types $\iota_A(x)=\iota_A(y)$ and $\iota_B(f(x))=\iota_B(f(y))$.
  By \ref{definition-containing-naturality} we can apply $\Im$ to maps and get a map $\Im f\colon \Im A \to \Im B$.
  So we can apply $\Im f$ to an equality $\gamma\colon \iota_A(x)=\iota_A(y)$ to get an equality
  \[\Im f (\gamma)\colon \Im f(\iota_A(x))=\Im f(\iota_A(y))\]
  Again by \ref{definition-containing-naturality}, we know that we have a naturality square:
  \begin{center}
    \begin{tikzcd}
      A\arrow[r,"\iota_A"]\arrow[d,swap,"f"] & \Im A\arrow[d,"\Im f"]\arrow[dl, Rightarrow, shorten <= 1em, shorten >= 1em,"\eta_f"] \\
      B\arrow[r,"\iota_B", swap] & \Im B 
    \end{tikzcd}
  \end{center}
  and hence equalities $\eta_f(x)\colon\Im f(\iota_A(x))= \iota_B(f(x))$ and $\eta_f(y)\colon \Im f(\iota_A(y))= \iota_B(f(y))$.
  This yields an equality of the desired type:
  \[ \eta_f(x)^{-1}\bu \Im f(\gamma) \bu \eta_f(y) \]
\end{proof}

A formal disk at a point is the ``collection'' of all other points infinitesimally close to it: 

\begin{definition}
  \label{def:formal-disk}
  Let $A$ be a type and $a\inT A$. 
  The type $\DD_a$ defined below in three equivalent ways is called the \notion{formal disk at $a$}.
  \begin{enumerate}[label=(\roman*)]
  \item $\DD_a$ is the sum of all points infinitesimally close to $a$, i.e.:
    \[ \DD_a\defequal\sum_{x\inT A}\iota_A(x)=\iota_A(a)\]
  \item $\DD_a$ is the fiber of $\iota_A$ at $\iota_A(a)$.
  \item $\DD_a$ is defined by the following pullback square:
    \begin{center}
      \begin{tikzcd}[auto]
        \DD_a\arrow[r]\arrow[d] & 1\arrow[d,"\ast\mapsto \iota_A(a)"]\arrow[dl, phantom, "\text{\scriptsize(pb)}" description] \\
        A\arrow[r,"\iota_A",swap] & \Im A 
      \end{tikzcd}
    \end{center}
  \end{enumerate}
\end{definition}
The characterization (iii) is a verbatim translation of its topos theoretic analog \cite{SchreiberDcct}[Definition 5.3.50] to homotopy type theory. 
Therefore, in the model from \cite{SchreiberDcct}, composing a function on a manifold $M$ with $\iota_M(x)$ would yield an $\infty$-order jet of that function. 
Jets are higher order analogues of tangent vectors and the derivates of a function at a point $x$ can still be observed on a formal disk around $x$.
To say that formal disks are just infinitesimal extensions of the point, is supported by the following observation.

\begin{proposition}
  For any $x:X$ we have $\Im(\DD_x)=1$.
\end{proposition}
\begin{proof}
  Using \cref{lemma-compute-modal-sum} and proposition \ref{lemma:retracts-pi-sigma} (d) we compute:
  \begin{align*}
    \DD_x&\equiv \sum_{y:X}\iota_X(x)=\iota_X(y) \\
         &=\sum_{y:X}\Im(\iota_X(x)=(\iota_X(y))) \\
         &=\Im\left(\sum_{z:\Im X}\iota_X(x)=z\right) \\
         &=1\rlap{.}
  \end{align*}
\end{proof}

As morphisms of manifolds induce maps on tangent spaces, 
maps of types induce morphisms on formal disks,
containing information on the derivates of a morphism of all orders:
\begin{remark}
  If $f: A\to B$ is a map, there is a dependent function:
  \[ df: \prod_{x\inT A} \DD_x\to\DD_{f(x)}\]
  We denote the evaluation at $a\inT A$ with
  \[ df_a: \DD_a \to \DD_{f(a)} \]
  and call it \notion{the (generalized) differential of $f$ at $a$}.
\end{remark}
\begin{proof}
  To define $df$ we take the sum over the map from \ref{mapping-preserves-infinitesimal-proximity}:
  \[ df_a\defequal (x , \epsilon) \mapsto (f(x) , \eta_f^{-1}(x)\bu\Im f (\epsilon) \bu\eta_f(x))\]
  -- where $\eta_f(x)$ is the equality from the naturality of $\iota$.
\end{proof}

Some of the familiar rules for differentiation can be derived in this generality.
We will need only the chain rule:
\begin{lemma}
  \label{chain-rule}
  Let $f:A\to B$ and $g:B\to C$ be maps. Then the following holds for all $x:A$
  \[ d(g\circ f)_x=(dg)_{f(x)}\circ df_x.\]
\end{lemma}
\begin{proof}
  Note that, in general, the differential $df_x$ is equal to the map induced by the universal property of $\DD_{f(x)}$ as a pullback.
  We can use this to get the desired ``functoriality'':
  \begin{center}
    \begin{tikzcd}
      \DD_x \arrow[d, dashed, "df_x"]\arrow[r]\arrow[dd, dashed, bend right=70, swap, "d(g\circ f)_x"] & A\arrow[r, "\iota_A"]\arrow[d, "f", swap] & \Im A \arrow[d, "\Im f"] \\
      \DD_{f(x)} \arrow[d, dashed, "(dg)_{f(x)}"]\arrow[r] & B\arrow[r, "\iota_B"]\arrow[d, "g", swap] & \Im B \arrow[d, "\Im g"] \\
      \DD_{g(f(x))} \arrow[r] & C\arrow[r, "\iota_C"] & \Im C
    \end{tikzcd}
  \end{center}
  -- the induced map $d(g\circ f)_x$ and the composition $(dg)_{f(x)}\circ df_x$ solve the same factorization problem, so they are equal.
\end{proof}

In differential geometry, the tangent bundle is an important basic construction consisting of all the tangent spaces in a manifold,
capturing first-order infinitesimal information.
In this abstract all-order setting, we can mimic the construction by combining all the formal disks of a space in a bundle, capturing infinitesimal information of all orders at once.
\begin{definition}
  \label{def-formal-disk-bundle}
  Let $A$ be a type.
  The type $T_\infty A$ defined in one of the equivalent ways below is called the \notion{formal disk bundle} of $A$.
  \begin{enumerate}[label=(\roman*)]
  \item $T_\infty A$ is the sum over all the formal disks in $A$:
    \[ T_\infty A\defequal\sum_{x\inT A} \DD_x \]
  \item $T_\infty A$ is defined by the following pullback square:
    \begin{center}
      \begin{tikzcd}[auto]
        T_\infty A\arrow[r]\arrow[d] & A\arrow[d,"\iota_A"]\drawpb{dl} \\
        A\arrow[r,"\iota_A",swap] & \Im A
      \end{tikzcd}
  \end{center}
  \end{enumerate}
\end{definition}
Note that despite the seemingly symmetric second definition,
we want $T_\infty A$ to be a bundle having formal disks as its fibers, so it is important to distinguish between the two projections and their meaning.
If we look at $T_\infty A$ as a bundle, meaning a morphism $p\colon T_\infty A\to A$,
we always take $p$ to be the first projection in both cases.
This convention agrees with the first definition -- taking the sum yields a bundle with 
fibers of the first projection equivalent to the formal disks.

For any $f\colon A \to B$ we defined the induced map $df$ on formal disks.
This extends to formal disk bundles.
\begin{definition}
  For a map $f\colon A\to B$ there is an induced map on the formal disk bundles, given as
  \[ T_\infty f\defequal (a , \epsilon) \mapsto (f(a), df_a(\epsilon))\]
\end{definition}

In differential geometry, the tangent bundle may or may not be trivial.
This is some interesting information about a space.
If we have a smooth group structure on a manifold $G$, i.e. a Lie-group,
we may consistently translate the tangent space at the unit to any other point.
This may be used to construct an isomorphism 
of the tangent bundle with the projection from the product of $G$ with the tangent space at the unit. 

It turns out that this generalizes to formal disk bundles 
and the group structure may be replaced by the weaker notion of a homogeneous type.

The notion of homogeneous type was developed by the author to satisfy two needs.
The first is to match the intuition of a pointed space, that is equipped with a continuous family of translations that map the base point to any given point.
The second need is to have just the right amount of data in all the proofs and constructions concerning homogeneous types.
It has not been investigated in what circumstances this definition of homogeneous spaces coincides with the various notions of homogeneous spaces in Geometry -- apart from the obvious examples given below.

\begin{definition}
  \label{def-homogeneous-type}
  A type $A$ is \notion{homogeneous}, if there are terms of the following types:
  \begin{enumerate}[label=(\roman*)]
  \item $e\inT A$
  \item $t\inT \prod_{x\inT A} A\simeq A$
  \item $p\inT \prod_{x\inT A} t_x(e)=x$
  \end{enumerate}
  Where $t$ is called the \notion{family of translations}
  and $e$ is called the \notion{unit} of $A$.
\end{definition}

\begin{examples}
  \begin{enumerate}
  \item Let $G$ be a group in the sense of \cite{UFP}[6.11], then $G$ is a homogeneous type with $x\bu\_$ or $\_\bu x$ as its family of translations.
  \item Let $G$ be an h-group, i.e. a type with a unit, operation and inversion that satisfy the group axioms up to a 2-cell.
    Then $G$ is a homogeneous type in the same two ways as above.
  \item As a notable special case, for any type $A$ and $\ast\inT A$, the loop space $\ast =_A \ast$ is homogeneous.
  \item Let $X$ be a connected H-space, then $X$ is homogeneous, again in two ways. See \cite{UFP}[8.5.2] and \cite{LicataFinster}[Section 4].
  \item Let $Q$ be a type with a quasigroup-structure, i.e. a binary operation $\_\bu\_$ such that all equations $a\bu x=b$ and $x\bu a=b$ have a contractible space of solutions,
    then $Q$ is homogeneous if it has a left or right unit. 
  \end{enumerate}
\end{examples}

In the following we will build a family of equivalences from one formal disk of a homogeneous type to any other formal disk of the space.
We start by observing how equivalences and equalities act on formal disks.
\begin{lemma}
  \label{equivalences-induce-equivalences-on-formal-disks}
  \begin{enumerate}
  \item
    If $f\colon A\to B$ is an equivalence, then
    \[ df_x\colon\DD_x\to \DD_{f(x)} \]
    is an equivalence for all $x\colon A$.
  \item
    Let $A$ be a type and $x,y\colon A$ two points.
    For any equality $\gamma\colon x=y$, we get an equivalence $\DD_x\simeq \DD_y$.
  \end{enumerate}
\end{lemma}
\begin{proof}
  \begin{enumerate}
  \item Let us first observe, that for any $x,y\colon A$ the map $\iota_A(x)=\iota_A(y)\to \iota_B(f(x))=\iota_B(f(y))$ is an equivalence.
  This follows from the fact that it is equal to the composition of two equivalences. 
  One is the conjugation with the equalities from naturality of $\iota$,
  the other is the equivalence of equalities induced by the equivalence $\Im f$. 

  Now, for a fixed $a\colon A$ we have two dependent types, $\iota_A(a)=\iota_A(x)$ and $\iota_B(f(a))=\iota_B(f(x))$
  and an equivalence over $f$ between them. 
  The sum of this equivalence over $f$ is by definition $df$ and by \ref{sum-equivalence-over} a sum of a fibered equivalence is an equivalence.
\item The equivalence is just the transport in the dependent type $x\mapsto \DD_x$.
  \end{enumerate}
\end{proof}

We are now ready to state and prove the triviality theorem.
\begin{theorem}
  \label{triviality-theorem}
  Let $V$ be a homogeneous type and $\DD_e$ the formal disk at its unit. 
  Then the following is true:
  \begin{enumerate}
  \label{triviality-theorem-trivialization}
  \item For all $x\inT V$, there is an equivalence
    \[ \psi_x\colon \DD_x\to \DD_e \]
  \item $T_\infty V$ is a trivial bundle with fiber $\DD_e$, i.e. we have an equivalence $T_\infty V\to V\times \DD_e$ and a homotopy commutative triangle
    \begin{center}
      \begin{tikzcd}
        T_\infty V\arrow[dr, "\pi_1",swap]\arrow[rr, "\simeq"] & & V\times \DD_e\arrow[dl, "\pi_1"] \\
        & V & 
      \end{tikzcd}
    \end{center}
  \end{enumerate}
\end{theorem}
\begin{proof}
  \begin{enumerate}
  \item Let $x\inT V$ be any point in $V$.
    The translation $t_x$ given by the homogeneous structure on $V$ is an equivalence.
    Therefore, we have an equivalence $\psi_x'\colon \DD_e\to \DD_{t_x(e)}$ by \ref{equivalences-induce-equivalences-on-formal-disks}.
    Also directly from the homogeneous structure, we get an equality $t_x(e)=x$ 
    and transporting along it yields an equivalence $\DD_{t_x(e)}\to \DD_{x}$.
    So we can compose and invert to get the desired $\psi_x$.
  \item By Definition \ref{def-formal-disk-bundle} of the formal disk bundle, we have 
    \[ T_\infty V\defequal\sum_{x\inT V}\DD_x\]
    We define a morphism $\varphi\colon T_\infty V\to V\times\DD_e$ by 
    \[ \varphi((x,\epsilon_x))\defequal (x, \psi_x(\epsilon_x))\]
    and its inverse by 
    \[ \varphi^{-1}((x,\epsilon_x))\defequal (x, \psi_x^{-1}(\epsilon_x))\text{.}\]
    Now, to see $\varphi$ is an equivalence with inverse $\varphi^{-1}$, one has to provide equalities of types
    \begin{align*}
      & (x,\epsilon_x)=\varphi^{-1}(\varphi(x,\epsilon_x))=(x,\psi^{-1}(\psi(\epsilon_x))) \\
      \text{ and } & (x,\epsilon_e)=\varphi(\varphi^{-1}(x,\epsilon_e))=(x,\psi(\psi^{-1}(\epsilon_e)))
    \end{align*}
    -- which exist since the $\psi_x$ are equivalences by (a).
  \end{enumerate}
\end{proof}

In geometry, it is usually possible to add tangent vectors.
Our formal disks can at least inherit the group-like properties of a homogeneous type:

\begin{theorem}
  \label{homogeneous-sequence}
  Let $A$ be homogeneous with unit $e:A$. Then $\DD_e$ is homogeneous.
\end{theorem}
\begin{proof}
  We look at the sequence
  \begin{center}
    \begin{tikzcd}
      \DD_e\arrow[r, "u_e"] & A\arrow[r, "\iota_A"] & \Im A
    \end{tikzcd}
  \end{center}
  where $u_e$ is the inclusion of the formal disk, given as the first projection.
  We will proceed by constructing a homogeneous structure on $\Im A$,
  note some properties of $\iota_A$ which could be part of a definition of \notion{morphism of homogeneous types}
  and finally give some ``kernel''-like construction of the structure on $\DD_e$.

  For $x:A$, there is a translation $t_x:A\simeq A$, since $\Im$ preserves equivalences, this yields a $\Im t_x: \Im A\simeq \Im A$.
  By $\Im$-elimination, this extends to a family of translations
  \[ t':\prod_{y:\Im A}\Im A\simeq \Im A \text{, with $t'_{\iota_A(x)}=\Im t_x$.} \]
  Let $e':\equiv\iota_A(e)$, then $\Im A$ is homogeneous if we can produce a
  \[ p':\prod_{y:\Im A}t'_y(e')=y. \]
  By $\Im$-eliminating on $y$, we reduce the problem to
  \[
    \prod_{x:A}t'_{\iota_A(x)}(\iota_A(e))=\iota_A(x)
  \]
  By definition, the left hand side is $\Im (t_x)(\iota_A(e))$
  and by naturality of $\iota$, we have an equality $\Im (t_x)(\iota_A(e))=\iota_A(t_x(e))$.
  So by applying $\iota_A$ to the equality $p_x:t_x(e)=x$, we get a solution.

  We start to construct the homogeneous structure on $\DD_e$ by letting $e'':\equiv (e,\refl)$ be the unit.
  For the translations, we look at the dependent type $(x:A)\mapsto \iota_A(e)=\iota_A(x)$ and establish the following chain of equivalences for $y:A$ with $\iota_A(e)=\iota_A(y)$:
  \begin{align*}
           &\ \iota_A(e)=\iota_A(x) \\
    \simeq &\ t'_{\iota_A(y)}\iota_A(e)=t'_{\iota_A(y)}\iota_A(x) \\
    \simeq &\ t'_{\iota_A(y)}\iota_A(e)=\iota_A(t_y(x))  \\
    \simeq &\ \iota_A(y)=\iota_A(t_y(x)) \\
    \simeq &\ \iota_A(e)=\iota_A(t_y(x))
  \end{align*}
  The resulting equivalence, is an equivalence over $t_y$.
  So by \ref{sum-equivalence-over} this induces an equivalence on the sum, which is $\DD_e$.

  This construction yields a family of equivalences $t'':\prod_{y:\DD_e}\DD_e\simeq \DD_e$.
  To finish the prove of the theorem, we need to construct a family of equalities $\prod_{x:\DD_e}t''_x(e'')=x$.
  This is another computation using the same methods we have seen so far and we refer to the \href{https://github.com/felixwellen/DCHoTT-Agda/blob/ca8c755af0b26f8f50c5a60d3b7f9384a26f5d0e/ImHomogeneousType.agda#L237-L241}{formalization}\footnote{\href{https://github.com/felixwellen/DCHoTT-Agda/blob/master/ImHomogeneousType.agda}{https://github.com/felixwellen/DCHoTT-Agda/blob/master/ImHomogeneousType.agda}} instead of giving the details here.
\end{proof}

\subsection{Formally étale maps}
\label{section-etale-maps}

The approach to formally étale maps presented here has been developed further
in the ongoing synthetic algebraic geometry project\footnote{\href{https://github.com/felixwellen/synthetic-zariski/blob/main/README.md}{https://github.com/felixwellen/synthetic-zariski/blob/main/README.md}}.

In algebraic geometry, formally étale maps are supposed to be analogous to local diffeomorphisms in differential geometry.
Below, we will give a not so well known definition which matches the notion of algebraic geometry for finitely presented morphisms of schemes 
\footnote{This is in the appendix of \cite{cherubini_rijke_2021}} 
and coincides with the local diffeomorphisms between manifolds in the case of differential geometry
\footnote{See \cite[Proposition 3.2]{SyntheticPDEs} for a precise statement in an intended model.}.

\begin{definition}
  \label{formally-etale}
  A map $f\inT A\to B$ is \notion{formally étale}, if its naturality square is a pullback:
  \begin{center}
    \begin{tikzcd}
      A\arrow[r, "\iota_A"]\arrow[d, "f", swap] & \Im A\arrow[d, "\Im f"] \\
      B\arrow[r, "\iota_B", swap] & \Im B
    \end{tikzcd}
  \end{center}
\end{definition}

To see why this definition expresses that a map is an isomorphism on a infinitesimal scale,
we can look at the following situation:
  \begin{center}
    \begin{tikzcd}
      1\ar[r]\ar[d] &A\arrow[r, "\iota_A"]\arrow[d, "f", swap] & \Im A\arrow[d, "\Im f"] \\
      \DD_b\ar[r]\ar[ru,dashed]   &B\arrow[r, "\iota_B", swap] & \Im B
    \end{tikzcd}
  \end{center}
 -- whenever we have a point $b:B$ which we can lift to $A$, there will be a unique lift of the whole formal disk around $b$ to $A$ by the universal property of the pullback $A$ and the naturality of $\Im$. 

This definition of formally étale maps was used extensively in \cite{SchreiberDcct} and
\cite{SyntheticPDEs}\footnote{One of the authors of these references, Schreiber, learned the definition from notes of Kontsevich and Rosenberg on Q-categories and communicated it to the author of this article.}.
The same definition under different names was also used in category theory to study the relation between reflective subcategories and factorization systems \cite{cassidy_herbert_kelly_1985}.
Here, the maps with a cartesian naturality square for the reflector, are the right maps of a factorization system, where the left maps are those mapped to isomorphisms by the reflector. The factorization system can also be defined for a modality and studied internally \cite{cherubini_rijke_2021}.

We will continue with some basic observations:

\begin{lemma}
  \label{lemma-basic-etale-properties}
  \begin{enumerate}
  \item If $f\colon A\to B$ and $g\colon B\to C$ are formally étale, their composition $g\circ f$ is formally étale.
    If the composition $g\circ f$ and $g$ are formally étale, then $f$ is formally étale.
  \item Equivalences are formally étale.
  \item Maps between $\Im$-modal types are formally étale.
  \item All fibers of a formally étale map are $\Im$-modal.
  \end{enumerate}
\end{lemma}
\begin{proof}
  \begin{enumerate}
  \item By pullback pasting.
  \item The naturality square for an equivalence is a commutative square with equivalences on opposite sides. 
    Those squares are always pullback squares.
  \item This is, again, a square with equivalences on opposite sides.
  \item The pullback square witnessing $f\colon A\to B$ being formally étale yields
    an equivalence over $\iota_B$. 
    So, each fiber of $f$ is equivalent to some fiber of $\Im f$.
    But fibers of maps between $\Im$-modal types are always $\Im$-modal by \ref{lemma:retracts-pi-sigma} (c),
    hence each fiber of $f$ is equivalent to a $\Im$-modal type, thus itself $\Im$-modal.
  \end{enumerate}
\end{proof}

Together with the following, we have all the properties of formally étale maps needed in this article:

\begin{lemma}
  \label{lemma-etale-pullback-square}
  Let $f\inT A\to B$ be formally étale, then the following is true:
  \begin{enumerate}
  \item For all $x\inT A$, the differential $df_x$ is an equivalence.
  \item There is a pullback square of the following form:
    \begin{center}
      \begin{tikzcd}
        T_\infty A\arrow[r]\arrow[d] & T_\infty B\arrow[d]\drawpb{dl} \\
        A\arrow[r, "f", swap] & B
      \end{tikzcd}
    \end{center}
  \end{enumerate}
\end{lemma}
\begin{proof}
  \begin{enumerate}
  \item The pullback square witnessing that $f$ is formally étale can be reformulated as: \\
    For all $x\inT \Im A$, the induced map between the fibers of $\iota_A$ and $\iota_B$ is an equivalence. 
    But these fibers are just the formal disks, so this can be applied to any $\iota_A(y)$ to see that $df_y$ is an equivalence.
  \item This is just a reformulation.
  \end{enumerate}
\end{proof}

One might wonder if the converse of this statement holds.
With a mild condition on $A$ which is related to the concept of formal smoothness in algebraic geometry,
this is the case\footnote{Formulation and proof of this remark are a result of a discussion with Hugo Moeneclaey.}:

\begin{remark}
  \label{formally-etale-by-differential}
  Let $A$ be a type such that $\iota_A:A\to\Im A$ is surjective and $f: A\to B$ any map.
  Then $f$ is formally étale, if $df_x$ is an equivalence for all $x\inT A$
\end{remark}
\begin{proof}
  As in the lemma, we use the equivalence of pullback squares and fibered equivalences.
  So to show that the $\iota$-naturality square for $f$ is a pullback,
  we have to show that for all $x\inT \Im A$ the induced map on fibers
  \[
    \psi_x:\iota_A^{-1}(x) \to \iota_B^{-1}((\Im f)(x))
  \]
  is an equivalence.

  By surjectivity of $\iota_A$, there merely is a $\tilde{x}$ and $p:\iota_A(\tilde{x})=x$.
  Since we show a proposition, we can use $\tilde{x}$ and the equivalence $e_1:\DD_{\tilde{x}}\simeq \iota_A^{-1}(x)$.
  By naturality we also have $e_2:\DD_{f(\tilde{x})}\simeq \iota_B^{-1}((\Im f)(x))$.
  
  It remains to show that $\psi_x=e_2\circ df_{ \tilde{x}}\circ e_1^{-1}$.
  Induction on $p$ simplifies $e_1$ and $e_2$ to the identity
  and we just have to note that $df_{ \tilde{x}}$ was defined as a induced map on the fibers $\DD_{ \tilde{x}}$ and $\DD_{f(\tilde{x})}$.
\end{proof}

The following is also proven in a different way in \cite{cherubini_rijke_2021} as corollary 5.2 (b).

\begin{theorem}
  \label{etale-is-pullback-stable}
  Let $f:A\to B$ be formally étale and
  \begin{center}
    \begin{tikzcd}
      A'\arrow[r]\arrow[d, "f'", swap] & A\arrow[d, "f"]\arrow[d]\drawpb{dl} \\
      B'\arrow[r] & B
    \end{tikzcd}
  \end{center}
  a pullback square. Then $f'$ is formally étale.
\end{theorem}

\begin{proof}
  Let us denote the bottom map with $\psi:B'\to B$.
  We start by describing $A'$ as a pullback:
  \begin{align*}
    A' &\simeq \left(\sum f'^{-1}\right)
    \simeq \left(\sum f^{-1}\circ\psi\right)
    \simeq \left(\sum (\Im f)^{-1}\circ\iota_B\circ\psi\right) \\
    & \simeq \left(\sum (\Im f)^{-1}\circ\Im\psi\circ\iota_{B'}\right)
  \end{align*}
  Now we can apply \ref{lemma-compute-modal-sum} to compute $\Im A'$:
  \[ \Im A'\simeq \Im\left(\sum (\Im f)^{-1}\circ\Im\psi\circ\iota_{B'}\right)
    \simeq\left(\sum (\Im f)^{-1}\circ\Im\psi\right) \]
  Note that the right hand side is the pullback of $\Im A$ along $\Im\psi$.
  This means that applying $\Im$ to the pullback square given in the statement of the theorem, is again a pullback and by pullback pasting
  the naturality square of $f'$ is a pullback.
\end{proof}

\begin{corollary}
  \begin{enumerate}
  \item Let $X$ be a type and $x:X$.
    The inclusion $\iota_x:\DD_x\to X$ of the formal disk at $x$ is a formally étale map.
  \item Any pullback of a map between $\Im$-modal types is formally étale.
  \end{enumerate}
\end{corollary}

\begin{proof}
  All maps between $\Im$-modal types are formally étale.
  Hence the second statement follows from the theorem and the first follows as the special case for the map $\iota_X(x):1\to \Im X$.
\end{proof}

There is much more to be said about formally étale maps that is very useful, but not used in this article.
One example which is interesting from a geometric perspective is that formally étale maps are the right class of a factorization system,
whose left class are the $\Im$-equivalences.
A consequence is that all maps can be factored into an $\Im$-equivalence followed by a formally étale map:

\begin{remark}
  Let $f:A\to B$ be a map.
  The map $f$ factors over
  \[
    C_f:\equiv \sum_{x\inT \Im A, y\inT B}(\Im f)(x)=\iota_B(y)
  \]
  by $l_f:\equiv (u:A)\mapsto (\iota_A(u),f(u),\eta_f)$ and $r_f:\equiv ((x,y,p)\mapsto y) $ to $B$.
  Furthermore, $\Im(l_f)$ is an equivalence and $r_f$ is formally étale.
\end{remark}

We sketch a proof -- a full analysis of the factorization system can be found in \cite[section 7]{cherubini_rijke_2021}.

\begin{proof}
  Applying \Cref{lemma-compute-modal-sum} twice on $C_f$ shows that $\Im(l_f)$ is an equivalence.
  And $r_f$ is formally étale, since it is the pullback of the formally étale $\Im f$ along $\iota_B$.
\end{proof}

We will put formally étale maps to use in section \ref{manifolds}.
The next section makes no reference to $\Im$.

\section[Structures on V-manifolds]{Structures on $V$-manifolds}
\subsection[Fiber bundles]{Fiber bundles}
\label{section-fiber-bundles}

As mentioned in the introduction, the spaces we have in mind might have both, differential geometric structure and higher identity types.
This section is about maps that correspond to fiber bundles which are by definition locally trivial with respect to the higher identity or homotopical structure,
but are expected to be locally trivial also with respect to a geometric structure which might be present in an application.
By local triviality, we just mean that there is a
surjection\footnote{Internal surjections correspond to effective epimorphisms in a topos, so in a topos of sheaves, the topology \emph{does} play a role for the internal surjections. }
into the base of the fiber bundle $p:E\to B$ such that pulling back along the surjection yields a projection from a product with a fixed given type $F$.
It will turn out to be logically equivalent to ask all fibers of $p$ to be merely equivalent to $F$.

In a basic intended applications $E$ and $B$ might just be manifolds given by 0-types and the reader might wonder if this notion of fiber bundle is too unrestrictive.
However, it turns out that asking this \emph{internally} turns into a surprisingly strong statement \emph{externally}.
In \cite{sag} it was discovered for a model based on a Grothendieck topos relevant to algebraic geometry,
that internal surjections have local sections with respect to the Grothendieck topology.
It is reasonable to assume that similar principles work in differential geometry and applying this to the surjective projection
\[ \left(\sum_{x:B}\|p^{-1}(x)=F\|\right) \to B \]
would show that the fiber bundles defined in this section are actually locally trivial in the sense of classical definitions of fiber bundles.
A $\infty$-topos-theoretic version of this approach to fiber bundles may be found in \cite{NikolausSchreiberStevenson}.

In this section, we will give four definitions of these fiber bundles\footnote{Some of the following definitions of fiber bundles were also used early in the short history of homotopy type theory at least by Mike Shulman, Ulrik Buchholtz and Egbert Rijke.} and prove that they are equivalent.
It will be useful in Section \ref{G-jet-structures} to switch between the different definitions.

For the following statements about fiber bundles, we will make a lot of unavoidable use of a univalent universe $\mathcal U$
and propositional truncation.
We will frequently use that all maps of types $p:E\to B$ appear in a pullback square
\begin{center}
  \begin{tikzcd}
    E\arrow[d, "p", swap]\arrow[r] & \widetilde{\mathcal U}\arrow[d]\drawpb{dl} \\
    B\arrow[r, "p^{-1}", swap] & \mathcal U,
  \end{tikzcd}
\end{center}
where $\widetilde{\mathcal U}$ is called the \notion{universal family} and obtained by summing over the dependent type $(A\inT \mathcal U)\mapsto A$.
The bottom map $p^{-1}$ determines $p$ up to canonical equivalence over $B$ and is called the \notion{classifying map} of $p$. 
If $E$ is a sum over a dependent type $q\inT B\to \mathcal U$, and $p$ the projection to $B$, then $q$ is the classifying map.

This way of using a univalent universe corresponds to looking at it as a \notion{moduli space} or \notion{classifying space} of types.
We could replace the $\mathcal U$ with some other moduli space to get specialized notions of fiber-bundles
with additional structure on the fibers. 

Before we start, we will look at some preliminaries about surjective and injective maps.
A surjective map is a map with merely inhabited fibers, or in other words a $\|\_\|$-connected map.
An injective map has $\|\_\|$-truncated fibers.
\footnote{Note that in a sheaf-topos, this notion corresponds to epimorphisms and not to a pointwise surjective map. 
  In \cite[chapter 7]{UFP}, surjective maps are called $(-1)$-connected or also surjective, if their domain and codomain are $0$-types. 
  Topos theoretic analogs are defined in \cite[6.5.1.10, 5.5.6.8]{Lurie} and are called 0-connective and (-1)-truncated.
In the terminology of \cite{SchreiberDcct} or \cite{nLab} and in \cite{wellen-thesis} surjective maps would be 1-epimorphisms and injective maps 1-monomorphisms.}

\begin{definition}
  Let $f:A\to B$ be a map of types.
  \begin{enumerate}
  \item The map $f$ is \notion{surjective} if
    \[ \prod_{b\inT B} \left(\| f^{-1}(b) \|\simeq 1\right)\text{.} \]
    We write $f: A\epi B$ in this case.
  \item The map $f$ is \notion{injective} if
    \[ \prod_{b\inT B} \left(f^{-1}(b) \text{ is a proposition}\right)\text{.} \]
    We write $f: A\mono B$ in this case.
  \end{enumerate}
\end{definition}

\begin{lemma}
  \label{surjective-pullback-stable}
  Surjective and injective maps are preserved by pullbacks.
\end{lemma}
\begin{proof}
  This is immediate by passing from pullback squares to fibered equivalences.
\end{proof}

\begin{examples}
  \label{example-equivalences-are-epi}
  \begin{enumerate}
  \item Let $f: A\to B$ be an equivalence of types. 
    Then $f$ is surjective and injective since all fibers of $f$ are contractible.
  \item Let $P: A\to \mathcal U$ be a proposition. Then the projection
    \[ \pi_1 : \sum_{a\inT A} P(a) \to A \]
    is injective.
  \item For the higher inductive type $S^1$, 
    the inclusion of the base point is a surjection.
  \end{enumerate}
\end{examples}

\begin{lemma}
  \label{lemma:epi-mono-factorization}
  For any map $f: A\to B$ there is a unique triangle:
  \begin{center}
    \begin{tikzcd}
      A\arrow[rr, "f"]\arrow[dr, twoheadrightarrow, "e", swap] & & B \\
      & \image(f)\arrow[ur, hookrightarrow, "m", swap] & 
    \end{tikzcd}
  \end{center}
  where $e$ is surjective, $m$ injective and $\image(f)$ is given by 
  \[\image(f)\defequal \sum_{b:B} \left\| \sum_{a\inT A} f(a)=b  \right\|\text{.} \]
\end{lemma}
A proof of the general case for $\|\_\|_n$ may be found in \cite[chapter 7.6]{UFP}. 

In Topology, an $F$-fiber bundle is a map $p\colon E\to B$ that is locally trivial with all its fibers are isomorphic to $F$.
Local triviality means that $B$ may be covered by open sets $U_i$, 
such that on each $U_i$ the restricted map $p_{|p^{-1}(U_i)}$ is isomorphic to the projection $F\times U_i\to U_i$. 
We may rephrase this in a more economical way: 
From our cover, we construct a surjective map $w\colon\coprod_{i\in I}U_i\to B$. 
Then the local triviality translate to the pullback of $p$ along $w$ being isomorphic to the product projection $F\times\coprod_{i\in I}U_i\to\coprod_{i\in I}U_i$.

For fiber bundles in geometry, we would require more from a general surjective map, or \notion{cover}, $w\colon W\to B$ than that pulling back along it turns $p$ into a product projection.
However, for the notion we discuss in this section, this turns out to be already enough.
\begin{definition}
  \label{def-fiber-bundle}
  Let $p\colon E\to B$ be a map of types. 
  For another map $w\colon W \to B$ we say $w$ is a \notion{trivialising cover} for $p$ if $w$ is a surjective map and there is a pullback square:
  \begin{center}
    \begin{tikzcd}
      W\times F\arrow[r]\arrow[d, "\pi_1", swap] & E\arrow[d, "p"]\drawpb{dl} \\
      W\arrow[r, "w", swap] & B
    \end{tikzcd}
  \end{center}
  The map $p$ is called an $F$-\notion{fiber bundle} if there merely is such a trivializing $p$.
\end{definition}

We give an equivalent dependent version of this definition,\footnote{Following a suggestion from Max S.\ New \url{http://maxsnew.github.io/}.}
which will be a lot easier to work with:

\begin{definition}
  \label{def-fiber-bundle-dependent}
  Let $E\inT B\to\mathcal U$ be a dependent type. 
  We say that a surjection $w: W \to B$ is a \notion{trivialising cover} for $E$ if 
  \[ \prod_{x\inT W} E(w(x)) \simeq F\text{.} \]
  The dependent type $E$ is called an $F$-\notion{fiber bundle} if there merely is such a trivialising cover.
\end{definition}

We can switch between the two definitions in the usual way: 
Given an $F$-fiber bundle $p\inT E\to B$ in the first sense, the dependent type of its fibers $p^{-1}\inT B \to \mathcal U$ 
will be an $F$-fiber bundle in the second sense, by direct application of \ref{rem-pullback-fiberwise-equiv}. 
To go back, we take the projection from the sum of an $F$-fiber bundle $E\inT B\to \mathcal U$.

Note that in both cases, the propositional truncation of the trivializing datum is neccessary to turn the definition into a proposition.
In the following, we will see that we could have defined $F$-fiber bundles more easily with their classifying maps to a type called $\BAut(F)$, providing us with a notion of $F$-fiber bundles, which is directly a proposition.  
However, in those definitions, while it is possible to construct a surjective trivializing map, it is unclear how we may require that this map has additional properties. 
One example, where we are interested in special surjections, is the definition of a $V$-manifold, where we will use formally étale surjections. 

We review the type $\BAut(F)$ now, which will be used to give the alternative definition of fiber bundles mentioned above:
\begin{definition}
  Let $F$ be a type and $t_F: 1\to \mathcal{U}$ the map given by $\ast \mapsto F$.
  \begin{enumerate}
  \item Let $\BAut(F)\defequal \image(t_F)$.
  \item We also have the injection $v_{\mathrm{BAut(F)}}: \BAut(F)\to \mathcal{U}$.
  \item We use the notation $F\hquot\Aut(F):\equiv \sum_{(F',\left|\varphi\right|)\inT\BAut(F)}F'$
    which is justified by the general fact that dependent sums over a map $\rho:BG\to \mathcal U$ are the homotopy quotient of $\rho(*)$ by the action of loops in $BG$ via transport in $\rho$.
  \end{enumerate}
\end{definition}

\begin{remark}
The first projection $\pi: F\hquot\Aut(F)\to \BAut(F)$ is a pullback of $\widetilde{\mathcal U}\to \mathcal U$ along $v_{\mathrm{BAut(F)}}$.
The map $\pi\inT F\hquot\Aut(F)\to \BAut(F)$ is the universal $F$-fiber bundle,
meaning all $F$-fiber bundles with any base will turn out to be pullbacks of this map.  
\end{remark}

We are now ready to give yet another definition of fiber bundles:
\begin{definition}
  \label{definition-classifying-square-for-fiber-bundle}
  A map $p: E\to B$ is an \notion{$F$-fiber bundle}, if and only if there is a map $\chi\colon B\to \BAut(F)$, 
  such that there is a pullback square
  \begin{center}
    \begin{tikzcd}
      E\arrow[d, "p", swap]\arrow[r] & F\hquot\Aut(F)\arrow[d, "\pi"]\drawpb{dl} \\
      B\arrow[r, "\chi", swap] & \BAut(F).
    \end{tikzcd}
  \end{center}
  In this case, $\chi$ is called the \notion{classifying map} of $p$.
\end{definition}

This definition also has a surprisingly easy dependent variant, which is obviously a mere proposition:

\begin{definition}
  \label{all-fibers-merely-equal}
  Let $E\colon B\to\mathcal U$ be a dependent type. 
  We say $E$ is an \notion{$F$-fiber bundle}, if
  \[ \prod_{b\inT B}\| E(b)\simeq F\| \text{.}\]
\end{definition}

Again, we will switch between the dependent and non-dependent version by taking fibers of $p$ and the sum respectively.
To arrive at the dependent version, we can directly use the classifying morphism $\chi$ of an $F$-fiber bundle $p\inT E \to B$ to construct a term of 
\[ \prod_{b\inT B}\| p^{-1}(b)\simeq F\|, \]
since all points $\chi(b)\inT \BAut(F)$ are of the form $(F',\gamma)$, 
with $F'\simeq p^{-1}(b)$ by the pullback square and $\gamma$ a proof that $F'$ is merely equivalent to $F$. 

Now, for the converse, let 
\[ E\inT B\to \mathcal U \]
be an $F$-fiber bundle, by $t\inT \prod_{b\inT B}\| E(b)\simeq F\|$. 
Then the classifying map is given by $(x\inT B)\mapsto (E(b), t_x)$ and the pullback square is given by pasting:
\footnote{Note that the outer rectangle is a pullback for all dependent types. }
\begin{center}
  \begin{tikzcd}
    \sum E\arrow[r]\arrow[d, "\pi_1", swap]\arrow[rr, bend left=25] & F\hquot\Aut(F)\arrow[d, "\pi", swap]\arrow[r] & \widetilde{\mathcal U}\arrow[d]\drawpb{dl} \\
    B\arrow[r, "\chi", swap] & \BAut(F)\arrow[r] & \mathcal U.
  \end{tikzcd}
\end{center}
We will conclude this section by showing that all our definitions of fiber bundles are equivalent and discuss some examples.
The equivalence is most efficiently proven, by establishing the equivalence of the two dependent definitions first:
\begin{theorem}
  \label{thm-equivalence-of-fiber-bundle-definitions}
  Let $F$ be a type and $E\inT B\to\mathcal U$ be a dependent type, then
  \[ \prod_{b\inT B}\| E(b)\simeq F\| \]
  if and only if there is a type $W$ and a surjective $w\inT W\to B$ such that
  \[ \prod_{x\inT W} E(w(x))\simeq F.\]
\end{theorem}
For the proof, we need to construct a trivializing cover at some point.\footnote{The author has to thank Ulrik Buchholtz for asking if such a cover always exists.}
The construction we use is similar to the universal cover and interesting on its own:
\begin{definition}
  \label{definition-canonical-cover}
  Let $E\inT B\to\mathcal U$ be an $F$-fiber bundle by $t\inT \prod_{b\inT B}\| E(b)\simeq F\|$, 
  then 
  \[ W\colon\equiv \sum_{b\inT B} E(b) \simeq F \]
  together with its projection to $B$ is the \notion{canonical trivializing cover} of $p$. 

  The given $t$ directly proves that this projection is surjective. 
  Let us denote this projection by $w\inT W\to B$, then for all $x\inT W$, with $x=(b,e)$ we have
  \[ E(w(x))\simeq E(\pi_1(b,e))\simeq F\]
  by transport and $e\inT E(b)\simeq F$ itself.
\end{definition}

\begin{proof}[of \ref{thm-equivalence-of-fiber-bundle-definitions}]
  With the definition and remark above, it remains to show the converse.
  Let $E\inT B\to\mathcal U$ and $w\inT W\to B$ such that $t\inT \prod_{x\inT W} E(w(x))\simeq F$.
  Now, for any $b\inT B$ and $x_b\inT w^{-1}(b)$, we get an equivalence $t_{\pi_1(x_b)}\inT E(w(\pi_1(x_b)))\simeq F$.
  By general properties of fibers, we have $w(\pi_1(x_b))=b$ yielding $E(b)\simeq F$.
  By surjectivity of $w$, we merely have a $x_b\inT w^{-1}(b)$ for any $b\inT B$, therefore we merely have an equivalence $E(b)\simeq F$.
\end{proof}

\begin{examples}
  \begin{enumerate}
  \item Let $A$ be a pointed connected type, then any $E\inT A\to\mathcal U$ is an $E(\ast)$-fiber bundle.\footnote{Thanks to Egbert Rijke for pointing this out.}
  \item The map $1\to S^1$ is a $\ZZ$-fiber bundle.
  \item More generally, for a pointed connected type $A$,
    projection from the homotopical universal cover $\sum_{x\inT A} x=\ast$ to $A$ is an $\Omega A$-fiber bundle
    and the projection from $\sum_{x\inT A}\| x=\ast \|_{1}$ to $A$ is a $\pi_1(A,\ast)$-fiber bundle.
  \item As $w\inT W\to B$ is a first projection, its fiber over any $b\inT B$ is equivalent to $E(b)\simeq F$.
    The latter type is merely equivalent to $\Aut(F)$, since $E(b)$ is merely equivalent to $F$.
    This means $w$ is an $\Aut(F)$-fiber bundle.
  \end{enumerate}
\end{examples}

\subsection[V-manifolds]{$V$-manifolds}
\label{manifolds}

A smooth $n$-manifold is a space that is locally diffeomorphic to $\RR^n$, hausdorff and second countable.
A detailed comparison between the notion of $V$-manifold, which will be introduced below, and other notions of manifold may be found in \cite{SyntheticPDEs}[3.3, 3.4] and \cite{david-orbifolds}[Section 5, p.\ 40 ff].

The definition of $V$-manifolds just mimics the property of being locally diffeomorphic to a fixed space,
which we will only require to be homogenous (as defined in \Cref{def-homogeneous-type}). 
A covering $(U_i)_{i\in I}$ with $U_i\simeq \RR^n$ of an $n$-manifold $M$ yields a surjective local diffeomorphism 
\[ \coprod_{i\in I}U_i \to M.\]
By projecting, there is also a local diffeomorphism from $U:\equiv \coprod_{i\in I}U_i$ to $\RR^n$.
So in total, we have a span of local diffeomorphisms where the right one is surjective.
  \begin{center}
    \begin{tikzcd}
      & U\arrow[dr, twoheadrightarrow]\arrow[dl, swap] & \\
      \RR^n & & M
    \end{tikzcd}
  \end{center}
In applications, more general vector spaces might take the role of $\RR^n$ -- so we will follow the literature and use the letter $V$ in the more abstract definition of a $V$-manifold below.
Instead of local diffeomorphisms we will use formally étale maps.
This is justified by the external calculation \cite{SyntheticPDEs}[Proposition 3.2] which shows that formally étale maps between two smooth manifolds are exactly the local diffeomorphisms.

\begin{definition}
  \label{def:V-manifold}
  Let $V$ be a homogeneous type. A type $M$ is a $V$-manifold if there is a span
  \begin{center}
    \begin{tikzcd}
      & U\arrow[dr, twoheadrightarrow, "\text{ét}"]\arrow[dl, "\text{ét}", swap] & \\
      V & & M
    \end{tikzcd}
  \end{center}
  where the left map is formally étale and the right map is formally étale and surjective.
\end{definition}

There is one trivial example:
\begin{example}
  Let $V$ be a homogeneous type, then $V$ is a $V$-manifold witnessed by the span:
  \begin{center}
    \begin{tikzcd}
      & V\arrow[dr, "\id"]\arrow[dl, "\id", swap] & \\
      V & & V
    \end{tikzcd}
  \end{center}
\end{example}

Less obvious are the following two ways of producing new $V$-manifolds.
However, without adding anything to our type theory making the modality $\Im$ more specific,
we cannot hope for examples that are not given as homogeneous types.
What could be added will be discussed at the beginning of the next section.

The statement in \ref{formal-disk-bundle-as-manifold} is a variant of the classical fact that the tangent bundle of a manifold is a manifold, 
but in our case, the infinitesimal information is kept separate.

\begin{lemma}
  \label{generate-manifolds}
  Let $V$ be homogeneous and $M$ be a $V$-manifold.
  \begin{enumerate}
  \item\label{formal-disk-bundle-as-manifold} The formal disk bundle $T_\infty M$ of $M$ is a $(V\times \DD_e)$-manifold.
  \item\label{etale-preserves-manifold} For any formally étale map $\varphi: N\to M$, $N$ is a $V$-manifold.
  \item\label{transitivity-of-being-V-manifold} If $V'$ is a homogeneous $V$-manifold and $N$ a $V'$-manifold, then $N$ is also a $V$-manifold.\footnote{This was a question by Ulrik Buchholtz.}
  \end{enumerate}
\end{lemma}
\begin{proof}
  \begin{enumerate}
  \item We can pull back the span witnessing that $M$ is a $V$-manifold along the projection $T_\infty M\to M$:
    \begin{center}
      \begin{tikzcd}
        V\times\DD_e\arrow[d]\drawpb{dr} & T_\infty U\arrow[d]\arrow[r, twoheadrightarrow, "\text{ét}"]\arrow[l, "\text{ét}", swap] & T_\infty M\arrow[d]\drawpb{dl} \\
        V & U\arrow[r, twoheadrightarrow, "\text{ét}", swap]\arrow[l, "\text{ét}"] & M
      \end{tikzcd}
    \end{center}
    By \ref{lemma-etale-pullback-square} (b) we know that the pullback of the map $T_\infty U\to U$ is the projection from the formal disk bundle of $U$.
    Formally étale maps are preserved by pullbacks by \ref{etale-is-pullback-stable} and surjective maps by \ref{surjective-pullback-stable},
    so the induced map $T_\infty U\to T_\infty M$ is formally étale surjective again.

    By \ref{triviality-theorem-trivialization} we know that $T_\infty V=V\times \DD_e$.
    So, again by \ref{lemma-etale-pullback-square} (b), we have the left pullback square.

    In \ref{homogeneous-sequence} we showed that $\DD_e$ is homogeneous,
    so $V\times\DD_e$ is homogeneous by giving it a componentwise structure.
  \item Pullback along $\varphi$ and composition give us the following:
    \begin{center}
      \begin{tikzcd}
          & \varphi^\ast U\arrow[d, swap, "\text{ét}"]\arrow[r, twoheadrightarrow, "\text{ét}"]\arrow[dl, "\text{ét}", swap, bend right=20] & N\arrow[d, "\varphi"]\drawpb{dl} \\
        V & U\arrow[r, twoheadrightarrow, "\text{ét}", swap]\arrow[l, "\text{ét}"] & M
      \end{tikzcd}
    \end{center}
  \item That $N$ is a $V$-manifold is witnessed by the following diagram using preservation of surjections and formally étale maps under pullbacks:
    \begin{center}
      \begin{tikzcd}
        &     & U_V\times_{V'}U_N\arrow[dr, twoheadrightarrow, "\text{ét}"]\arrow[dl, "\text{ét}", swap]\drawpb{dd} &     &   \\
        & U_V\arrow[dr, twoheadrightarrow, "\text{ét}"]\arrow[dl, "\text{ét}", swap] &                   & U_N\arrow[dr, twoheadrightarrow, "\text{ét}"]\arrow[dl, "\text{ét}", swap] &   \\
      V &     &      V'           &     & N
    \end{tikzcd}
  \end{center}
  \end{enumerate}
\end{proof}

One important special case of part \ref{etale-preserves-manifold} of the lemma is that any formal disk $\DD_x$ of $M$ is a $V$-manifold.

In the following, let $V$ be homogeneous and $M$ be a fixed $V$-manifold.
The definition of $V$-manifolds entails a stronger local triviality condition on the formal disk bundle of $M$ than was discussed in the last section about $F$-fiber bundles, 
since there has to be a formally étale trivialising cover
\footnote{``Trivializing cover'' was defined in \Cref{def-fiber-bundle} and \Cref{def-fiber-bundle-dependent}}.
This property of the trivialising cover will \emph{not} be used in the following lemma.

\begin{lemma}
  \label{classifying-lemma-for-manifolds}
  \begin{enumerate}    
  \item \label{lemma-trivializing-square-for-TM}
    The formal disk bundle of the covering $U$ is trivial and there is a pullback square:
    \begin{center}
      \begin{tikzcd}
        U\times \DD_e\arrow[d]\arrow[r] & T_\infty M\arrow[d]\drawpb{dl} \\
        U\arrow[r, twoheadrightarrow] & M
      \end{tikzcd}
    \end{center}
  \item \label{lemma-classifying-square-for-TM} 
    The formal disk bundle of $M$ has a \notion{classifying morphism} $\tau\inT M\to \BAut(\DD_e)$, i.e. there is a pullback square:
    \begin{center}
      \begin{tikzcd}
        T_\infty M\arrow[r]\arrow[d] & \DD_e \hquot\Aut(\DD_e)\arrow[d, "\pi"]\drawpb{dl} \\
        M\arrow[r, "\tau_M", swap] & \BAut(\DD_e)
      \end{tikzcd}
    \end{center}
  \end{enumerate}
\end{lemma}

\begin{proof}
  \begin{enumerate}
  \item By \ref{lemma-etale-pullback-square}, there is a pullback square for the formally étale map to $V$:
    \begin{center}
      \begin{tikzcd}
        T_\infty U\arrow[r]\arrow[d] & T_\infty V\arrow[d]\drawpb{dl} \\
        U\arrow[r] & V
      \end{tikzcd}
    \end{center}
    Since $V$ is homogeneous, by \ref{triviality-theorem} its formal disk bundle is trivial. This is preserved by pullback, so $T_\infty U$ is trivial.
    The pullback square in the proposition is again given by \ref{lemma-etale-pullback-square}.
  \item The statement \ref{lemma-trivializing-square-for-TM} tells us, that $T_\infty M$ is a $\DD_e$-fiber bundle by Definition \ref{def-fiber-bundle}.
    And \ref{lemma-classifying-square-for-TM} is just another way to state that fact, namely Definition \ref{definition-classifying-square-for-fiber-bundle}.
  \end{enumerate}
\end{proof}

The classifying morphism $\tau_M$ is compatible with formally étale maps in the sense of the following remark.

\begin{remark}
  \label{differential-as-2-cell}
  Let $\varphi:N\to M$ be formally étale, then $N$ is also a $V$-manifold by \ref{generate-manifolds}.
  There is a 2-cell given by the differential of $\varphi$:
  \begin{center}
    \begin{tikzcd}
      |[alias=M]|M\arrow[r, "\tau_M"] & \BAut(\DD_e) \\
      N\arrow[u, "\varphi"]\arrow[ru, "\tau_N"{name=T}, bend right=35, swap]
      \arrow[Rightarrow, from=T, to=M, "d\varphi", swap, shorten <= 1em, shorten >= 1em]
    \end{tikzcd}
  \end{center}
\end{remark}
\begin{proof}
  We proved in \Cref{lemma-etale-pullback-square} (a) that the differential of a formally étale map
  is an equivalence at all points. Applied to $\varphi$, this fact may be expressed in the following way:
  \[ d\varphi:\prod_{x:N}\DD_x\simeq\DD_{\varphi(x)} \]
  This yields a 2-cell of the desired type, since the formal disks $\DD_x$ and $\DD_{\varphi(x)}$ are merely equivalent to $\DD_e$ for all $x$.
\end{proof}

This will be useful when we work with $G$-jet-structures in the next section.

\subsection[G-jet-structures]{$G$-jet-structures}
\label{G-jet-structures}

Intuitively, the classifying morphism $\tau_M:M\to\BAut(\DD_e)$ of a $V$-manifold $M$ describes how the formal disk bundle is glued together using automorphisms of $\DD_e$.
Lifts of $\tau_M$ along the delooping $\mathrm{B}G\to\BAut(\DD_e)$ of a morphism from a group $G$ will be called $G$-jet-structures. 

Some simple, classical $G$-structures on $\RR^n$-manifolds (or $n$-manifolds) only consider automorphisms of the tangent space,
so their delooped automorphism group $\mathrm{B}\GL_n(\RR)$
\footnote{Note that $\GL_n(\RR)$ and other groups appearing in the table below are \emph{set-level} structures in the intended applications,
so there is no problem with defining them.}
takes the role of $\BAut(\DD_e)$ in our $G$-jet-structures.
For an $\mathbb{R}^n$-manifold, the type $\BAut(\DD_e)$ will be a delooping of the infinite jet group $\mathrm{J}^\infty_n(\RR)$.
It is known (see for example \cite[p. 131]{NaturalDiffGeo}),
that the kernel of the projection $\mathrm{J}^\infty_n(\RR)\to \GL_n(\RR)$ is contractible.
The projection also has a section given by extending linear automorphisms to the formal disk.
This situation is nice enough, that we expect no problems with lifting our general classifying map $\tau_M:M\to\BAut(\DD_e)$,
to a classical classifying map $M\to \mathrm{B}\GL_n(\RR)$ in the case of $\RR^n$-manifolds --
which would admit reusing the classical examples.

There are lots of interesting classical examples of structures on manifolds that can be encoded as $G$-structures.
We give a list of examples, what group morphisms -- which are almost always inclusions of subgroups --
encode structures on a smooth $n$-manifold as $G$-structures. 
Some of the examples assume $n=2d$.
\begin{center}
  \label{G-structure-table}
  \begin{tabular}{ll}
    \toprule
    $\mathrm{G}\to \GL(n)$ & $G$-structure  \\
    \midrule
    $\mathrm{O}(n)\to\GL(n)$ & Riemannian metric \\
    $\GL^{+}(n)\to \GL(n)$  & orientation \\
    $\mathrm{O}(n-1,1)\to\GL(n) $ & pseudo-Riemannian metric \\ 
    $\mathrm{SO}(n,2)\to\GL(n)$ & conformal structure \\
    $\GL(d,\CC)\to \GL(2d,\RR)$ & almost complex structure \\
    $\mathrm{U}(d)\to \GL(2d,\RR)$ & almost Hermitian structure \\
    $\mathrm{Sp}(d)\to \GL(2d,\RR)$ & almost symplectic structure \\
    $\mathrm{Spin}(n)\to\GL(n)$ & spin structure \\
    \bottomrule
  \end{tabular}
\end{center}
For a definition of $\mathrm{O}(n)$- and $\GL(d,\CC)$-structures, see \cite{Chern}.
Note that in all of the above examples, $G$ is a 1-group\footnote{We call a 0-type with a group structure a 1-group.},
yet our theory also supports higher groups. 
The string 2-group and the fivebrane 6-group are examples of higher $G$-structures of interest in physics. 
See \cite{SatiSchreiberStasheffFivebrane} for details and references.
In the classical theory torsion-free $G$-structures are to $G$-structures what symplectic structures are to almost symplectic structures.
We will give a candidate analog of torsion-freeness for $G$-jet-structures at the end of this section.

We will now turn to the formal treatment of $G$-jet-structures on $V$-manifolds and the construction of the moduli spaces of these structures.
From now on, let $V$ be a homogeneous type.
As we learned in the last section in \ref{classifying-lemma-for-manifolds}, the formal disk bundle of a $V$-manifold $M$ is always
classified by a morphism $\tau_M:M\to\BAut(\DD_e)$, where $\DD_e$ is the formal disk at the unit $e:V$.
Since this is the only feature of a $V$-manifold that we need for the constructions in this section,
we will work with the following more general class of spaces,
where $D$ is an arbitrary type which takes the role of $\DD_e$.

\begin{definition}
  A type $M$ is called \notion{formal $D$-space}\footnote{The name was invented by Urs Schreiber and the author for the present purpose.}
 if its formal disk bundle is a $D$-fiber bundle.
\end{definition}

\begin{remark}
  \begin{enumerate}
  \item Any $V$-manifold $M$ is a formal $\DD_e$-space.
  \item Being a formal $D$-space is a proposition.
  \end{enumerate}
\end{remark}
\begin{proof}
  \begin{enumerate}
  \item This is \ref{classifying-lemma-for-manifolds}.
  \item One of the equivalent definitions of $D$-fiber bundle, \ref{all-fibers-merely-equal}, was directly a proposition:
    \[ (\text{$P:A\to\mathcal U$ is a $D$-fiber bundle}):\equiv\prod_{x:A}\|P(x)\simeq D\|\]
  \end{enumerate}
\end{proof}
We are interested in the case $D\equiv\DD_e$ for $e:V$ meaning that $M$ is a formal $\DD_e$-space if $\prod_{x:M}\|\DD_x\simeq\DD_e\|$.
In \ref{generate-manifolds} we saw, that we can ``pullback'' the structure of a $V$-manifold along a formally étale map.
Formal $\DD_e$-spaces behave the same way by virtue of the 2-cell we already saw in \ref{differential-as-2-cell}.
\begin{lemma}
  Let $M$ be a formal $\DD_e$-space.
  For any formally étale $\varphi:N\to M$, $N$ is also a formal $\DD_e$-space and there is the triangle:
  \begin{center}
    \begin{tikzcd}
      |[alias=M]|M\arrow[r, "\tau_M"] & \BAut(\DD_e) \\
      N\arrow[u, "\varphi"]\arrow[ru, "\tau_N"{name=T}, bend right=35, swap]
      \arrow[Rightarrow, from=T, to=M, "d\varphi", swap, shorten <= 1em, shorten >= 1em]
    \end{tikzcd}
  \end{center}
\end{lemma}
\begin{proof}
  First, the triangle in the statement exists for a formally étale map between any types, if $\BAut(\DD_e)$ is replaced with the universe:
  \begin{center}
    \begin{tikzcd}
      |[alias=M]|M\arrow[r, "x\mapsto\DD_x"] & \mathcal U \\
      N\arrow[u, "\varphi"]\arrow[ru, "x\mapsto\DD_x"{name=T}, bend right=35, swap]
      \arrow[Rightarrow, from=T, to=M, "d\varphi", swap, shorten <= 1em, shorten >= 1em]
    \end{tikzcd}
  \end{center}
  By assumption we know, that $(x:M)\mapsto \DD_x$ lands in $\BAut(\DD_e)$.
  But $\varphi$ is formally étale, so we have $d\varphi:\prod_{x:N}\DD_x\simeq\DD_{f(x)}$.
  The latter may be truncated and composed with $\tau_M:\prod_{x:M}\|\DD_x\simeq\DD_e\|$ to get $\tau_N:\prod_{x:N}\|\DD_x\simeq\DD_e\|$.
  So both maps to $\mathcal U$ factor over $\BAut(\DD_e)$.
\end{proof}
From now on, we assume that $\BG$ is a connected, pointed type and $(\ast=_{\BG}\ast)\simeq\mathrm{G}$.
We will define $G$-jet-structures\footnote{See the introduction to this section \ref{G-jet-structures}, for an explanation of the appearance of ``jets'' at this point.} or \notion{reductions of the structure group}, a synonym hinting that in a lot of cases, $G$ is a subgroup of $\Aut(\DD_e)$.
We will not restrict ourselves to reductions to subgroups and look at general pointed maps $\BG\to\BAut(\DD_e)$,
which correspond to general group homomorphisms $\mathrm{G}\to\Aut(\DD_e)$.
\begin{definition}
  \label{def:G-jet-structure}
  Let $\chi:\BG\to\BAut(\DD_e)$ be a pointed map and $M$ be a formal $\DD_e$-space.
  A \notion{$G$-jet-structure} on $M$ is a map $\varphi:M\to\BG$ together with a 2-cell $\eta:\chi\circ\varphi\Rightarrow\tau_M$:
  \begin{center}
    \begin{tikzcd}
      & |[alias=BG]|\BG\arrow[d, "\chi"] \\
      M\arrow[ur, "\varphi", bend left=30]\arrow[r, "\tau_M"{name=tau}, swap] & \BAut(\DD_e)
      \arrow[Rightarrow, from=BG, to=tau, shorten <= 0.5em, shorten >= 1em, "\eta", swap]
    \end{tikzcd}
  \end{center}
  We write
  \[ \Gstr(M):\equiv\sum_{\varphi:M\to\BG}(\chi\circ\varphi\Rightarrow\tau_M) \]
  for the type of $G$-jet-structures on $M$.
\end{definition}
The special case $G=1$ turns out to be interesting -- a $1$-jet-structure on a formal $\DD_e$-space is nothing else than a trivialization of the formal disk bundle,
like we produced in \ref{triviality-theorem-trivialization} for any homogeneous type.
This provides us with an example of a $1$-jet-structure, whose construction is, in spite of the name we will give below, not entirely trivial.
\begin{definition}
  The \notion{trivial $1$-jet-structure} on $V$ is the trivialization $\psi:\prod_{x:V}\DD_e\simeq\DD_x$ constructed in \ref{triviality-theorem-trivialization}:
  \begin{center}
    \begin{tikzcd}
      & |[alias=B1]|\mathrm{B}1:\equiv 1\arrow[d, "\ast\mapsto\DD_e"] \\
      V\arrow[ur, "\_\mapsto\ast", bend left=30]\arrow[r, "\tau_V"{name=tau}, swap] & \BAut(\DD_e)
      \arrow[Rightarrow, from=B1, to=tau, shorten <= 1em, shorten >= 1em, "\psi", swap]
    \end{tikzcd}
  \end{center}
\end{definition}
Since we have pointed maps, there is a triangle for any $\chi:\BG\to\BAut(\DD_e)$:
\begin{center}
  \begin{tikzcd}
    \mathrm{B}1\arrow[r, "\ast\mapsto\ast"]\arrow[d, swap, "\ast\mapsto\DD_e"{name=m}] & |[alias=BG]|\BG\arrow[dl, bend left=20, "\chi"] \\
    \BAut(\DD_e)\arrow[Rightarrow, from=m, to=BG, shorten <= 1.5em, shorten >= 1.5em]
  \end{tikzcd}
\end{center}
So we can define a trivial structure in the same way as above for arbitrary $G$.
Let us fix a pointed map $\chi:\BG\to\BAut(\DD_e)$ from now on.
\begin{definition}
  \label{definition-of-trivial-G-structure}
  Let $\mathcal T:\DD_e\simeq \chi(\ast)$ be the transport along the equality witnessing that $\chi$ is pointed.
  The \notion{trivial $G$-jet-structure} on $V$ is given by $\psi'_x:\equiv \psi_x\circ\mathcal T$:
  \begin{center}
    \begin{tikzcd}
      & |[alias=BG]|\BG\arrow[d, "\chi"] \\
      V\arrow[ur, "\_\mapsto\ast", bend left=30]\arrow[r, "\tau_V"{name=tau}, swap] & \BAut(\DD_e)
      \arrow[Rightarrow, from=BG, to=tau, shorten <= 1em, shorten >= 1em, "\psi'", swap]
    \end{tikzcd}
  \end{center}
\end{definition}
An important notion that we will introduce in the end of this section, is a \emph{torsion-free} $G$-structure.
In some sense to be made precise, these $G$-jet-structures will be trivial on all formal disks.
Before we can do this, we need to be able to restrict $G$-jet-structures to formal disks, or more generally, to pull them back along formally étale maps.
\begin{definition}
  \begin{enumerate}
  \item For $M$ a formal $\DD_e$-space and $f:N\to M$ a formally étale map from some type $N$,
    there is a map $f^\ast:\Gstr(M)\to\Gstr(N)$.
  \item For the special case of formal disk inclusions $u_x:\DD_x\to M$ and $\Theta:\Gstr(M)$,
    we call $ u_x^\ast\Theta$ 
    the \notion{restriction} of $\Theta$ to the formal disk at $x$.
  \end{enumerate}
\end{definition}
\begin{construction}[of $f^\ast$]
  \label{construction-of-pullback-structure}
  Let $\Theta\equiv(\varphi,\eta):\Gstr(M)$.
  Then we can paste the triangle constructed in \ref{differential-as-2-cell}
  to the triangle given by $(\varphi,\eta)$:
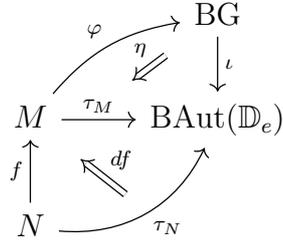
  \begin{center}
    \begin{tikzcd}
      & |[alias=BG]|\BG\arrow[d, "\chi"] \\
      |[alias=M]|M\arrow[r, "\tau_M"{name=tau}]\arrow[ru, "\varphi", bend left=20] & \BAut(\DD_e) \\
      \arrow[Rightarrow, from=BG, to=tau, shorten <= 1em, shorten >= 0.6em, "\eta", swap]
      N\arrow[u, "f"]\arrow[ru, "\tau_N"{name=T}, bend right=35, swap]
      \arrow[Rightarrow, from=T, to=M, "df", swap, shorten <= 1em, shorten >= 1em]
    \end{tikzcd}
  \end{center}
  We define the result of the pasting to be $f^\ast(\varphi,\eta):\Gstr(N)$.
  Or, put differently:
  \[ f^\ast(\varphi,\eta):\equiv(\varphi\circ f,(y:N)\mapsto\eta_{f(y)}\bu df_y^{-1}). \]
\end{construction}
Pulling back $G$-jet-structures is 1-functorial in the following sense.
\begin{remark}
  \label{pullback-G-jet-structures-1-functorial}
  Let $f:N\to M$, $g:L\to N$ be formally étale and $M$ a formal $\DD_e$-space then there is a triangle
  \begin{center}
    \begin{tikzcd}
      \Gstr(M)\arrow[rr, "(f\circ g)^\ast"{name=m}]\arrow[dr, "f^\ast", swap] & & \Gstr(L) \\
      & |[alias=N]|\Gstr(N)\arrow[ur, "g^\ast", swap] &
      \arrow[Rightarrow, from=m, to=N, shorten <= 0.7em, shorten >= 0.7em]
    \end{tikzcd}
  \end{center}
\end{remark}
\begin{proof}
  By \ref{chain-rule} we have 
  \[ d(f\circ g)_x=(df)_{g(x)}\circ dg_x.\]
  In diagrams, this yields a 3-cell between the pasting of
  
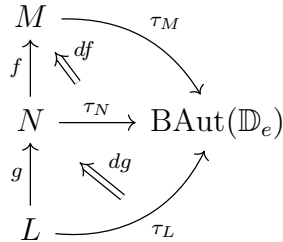
\begin{center}
    \begin{tikzcd}
      |[alias=M]|M\arrow[dr, "\tau_M", bend left=30] & \\
      |[alias=N]|N\arrow[u, "f"]\arrow[r, "\tau_N"{name=TN}] & \BAut(\DD_e) \\
      L\arrow[u, "g"]\arrow[ru, "\tau_L"{name=TL}, bend right=35, swap]
      \arrow[Rightarrow, from=TL, to=N, "dg", swap, shorten <= 1em, shorten >= 1em]
      \arrow[Rightarrow, from=TN, to=M, "df", swap, shorten <= 0.5em, shorten >= 0.5em]
    \end{tikzcd}
  \end{center}
  and 
  
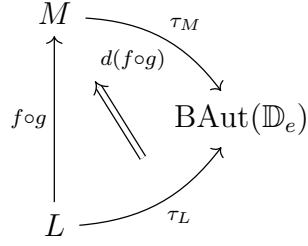
\begin{center}
    \begin{tikzcd}
      |[alias=M]|M\arrow[dr, "\tau_M", bend left=25] & \\
      & \BAut(\DD_e) \\
      L\arrow[uu, "f\circ g"]\arrow[ur, "\tau_L"{name=TL}, swap, bend right=24] & 
      \arrow[Rightarrow, from=TL, to=M, "d(f\circ g)", swap, shorten <= 1.7em, shorten >= 1.7em, near end]
    \end{tikzcd}
  \end{center}
  This means the 2-cells we paste when applying $(f\circ g)^\ast$ or $g^\ast\circ f^\ast$ are equal,
  so the functions must be equal, too.
\end{proof}
Let $M$ be a fixed formal $\DD_e$-space from now on.
The final definition of this article is that of a torsion-free $G$-jet-structure.
The aim is to ask if a $G$-jet-structure ``looks like the trivial $G$-jet-structure everywhere on an infinitesimal scale''.
To do this we restrict a $G$-jet-structure to the formal disk at a point and compare it to the trivial $G$-jet-structure on $\DD_e$.
So let us fix a notation for this structure:
\begin{definition}
  Let $\xi:\Gstr(V)$ be the trivial $G$-jet-structure from \ref{definition-of-trivial-G-structure} and $ u_e:\DD_e\to V$ the formal disk inclusion.
  Then
  \[ \xi_e:\equiv  u_e^* \xi \]
  is the \notion{trivial $G$-jet-structure} on $\DD_e$.
\end{definition}

But a priori, we have no means of comparing $G$-jet-structures on formal disks with this trivial structure,
so we need formally étale maps from all formal disks to $\DD_e$.
For formal $\DD_e$-spaces we merely have an equivalence from any formal disk to $\DD_e$.
More precisely, by \ref{all-fibers-merely-equal} we have
\[ \tau_M:\prod_{x:M}\|\DD_x\simeq \DD_e\|. \]
And by pulling back to the canonical cover $w:W\to M$ from \ref{definition-canonical-cover} we get
\[ \omega_M:\prod_{x:W}\DD_{w(x)}\simeq \DD_e\]
This is enough to make the indicated comparison.
\begin{definition}
  \label{def-torsion-free}
  A $G$-jet-structure $\Theta$ on $M$ is \notion{torsion-free},
  if
  \[ \prod_{x:W} \| (\omega_{M,w(x)}^{-1})^*u_{w(x)}^*\Theta = \xi_e\|\equiv:\text{torsion-free}(\Theta) \]
\end{definition}

It turns out that even for the trivial 1-jet-structure on $V$, torsion-freeness is non-trivial.
If the trivial 1-jet-structure is \emph{left-invariant} as defined below, it is an example of a torsion-free 1-jet-structure.
To match classic notions, we assume that the equivalences of the homogeneous structure are left-translations.
\begin{definition}
  The trivial $G$-jet-structure $\xi$ on $V$ is called \notion{left-invariant}, if the following condition holds:
  \begin{align*}
    \prod_{x:V}t_x^\ast\xi=\xi
  \end{align*}
\end{definition}
If our homogeneous space $V$ is a Lie-Group, the trivial 1-jet-structure is constructed the same way as the Maurer-Cartan form,
which satisfies the equation above. Turning this around, we get the following example:\footnote{This example and its presentation are a result of a discussion with Urs Schreiber.
}

\begin{theorem}
  \label{example-left-invariance-group}
  Let $V$ be a 1-group and its homogeneous structure be given by left translations,
  then the trivial $G$-jet-structure given by this homogeneous structure is left-invariant. 
\end{theorem}

\begin{proof}
  We will use the following equation given by the group structure:
  \[ t_{t_x(y)}=t_{xy}=t_x\circ t_y\]
  Evaluating at $e$ and using the chain rule \ref{chain-rule} yields:
  \[ d(t_{t_x(y)})_e=\left(dt_x\right)_{t_y(e)}\circ \left(dt_y\right)_e=\left(dt_x\right)_y\circ (dt_y)_e=(dt_y)_e\bu\left(dt_x\right)_y \]
  The latter equality is just moving our equation to $\BAut(\DD_e)$.

  Now for the trivial $G$-jet-structure $\xi\equiv \left(\_\mapsto\ast, y\mapsto\left(dt_y\right)_e\right)$ we can calculate
  \begin{align*}
    t_x^\ast\xi&=\left(\_\mapsto\ast, y\mapsto\left(dt_{t_x(y)}\right)_e\bu \left(dt_x\right)_y^{-1}\right) \\
                  &=\left(\_\mapsto\ast, y\mapsto\left( \left( dt_y \right)_e\bu\left( dt_x\right)_y \right)\bu \left(dt_x\right)_y^{-1}\right) \\
                  &=\left(\_\mapsto\ast, y\mapsto\left(dt_y \right)_e\right) \\
                  &=\xi
  \end{align*}
\end{proof}

\begin{theorem}
  Let $V$ be a homogeneous space such that the trivial $G$-jet-structure is left-invariant, then the trivial $G$-jet-structure on $V$ is torsion-free.
\end{theorem}

\begin{proof}
  Let $t_x$ be the translation to $x:V$ given by the homogeneous structure on $V$ and
  $\xi\equiv(\_\mapsto\ast, x\mapsto\left(dt_x\right)_e)$ the trivial $G$-jet-structure on $V$.
  Then for all $x:V$ we have a square of formally étale maps:
  \begin{center}
    \begin{tikzcd}
      \DD_{x}\arrow[r, " u_{x}"] & V \\
      \DD_e\arrow[r, " u_e", swap]\arrow[u, "dt_x"] & V\arrow[u, "t_x", swap]
    \end{tikzcd}
  \end{center}
  By \ref{pullback-G-jet-structures-1-functorial}, we get the following formula:
  \[  u_e^\ast t_x^\ast\xi=dt_x^\ast  u_{x}^\ast\xi\]
  By \ref{example-left-invariance-group} we can simplify the left hand side:
  \[  u_e^\ast \xi=dt_x^\ast  u_{x}^\ast\xi\]
  The left hand side is the trivial structure on $\DD_e$ and
  we have to identify the right hand side with the term $(\omega_{M,x}^{-1})^\ast  u_{\omega(x)}^\ast\xi$ from \ref{def-torsion-free},
  where $\omega=\id_V$.
  This amounts to an identification $dt_x^\ast  u_{x}^\ast\xi=u^\ast_x\xi$,
  which is given by \ref{triviality-theorem-trivialization}.
\end{proof}
Since torsion-freeness -- as we defined it -- is a proposition,
the type of torsion-free $G$-jet-structures is a subtype of the type of $G$-jet-structures.
The latter should be distinguished from the \notion{moduli space} of $G$-jet-structures on $M$,
which is the quotient of the type of $G$-jet-structures by the action of the automorphism group of $M$.
If $M$ is a 0-type, we could just build this quotient as a higher inductive type, but this is a bit unsatisfactory and not the most pleasant definition to work with.
A more promising approach is to use
that the quotient of an action given as a dependent type $\rho:\BG\to\mathcal U$ is just $\sum_{x:\BG}\rho(x)$.
To make this approach work, the author reformulated a lot of the original theory in \cite{wellen-thesis}.
With the present version, we will see that this construction works without considerable effort.

To realize the construction of the moduli space as a dependent sum,
we need to note, that the definition of $G$-jet-structures is actually a dependent type over $\BAut(M)$.

\begin{lemma}
  There is a dependent type $\Gstr:\BAut(M)\to\mathcal U$ with $\Gstr(M')$ being the $G$-jet-structures on $M'$.
\end{lemma}

\begin{proof}
  Since any $M':\BAut(M)$ is equivalent to $M$, it is merely a formal $\DD_e$-space.
  Being a formal $\DD_e$-space is a proposition,
  so $\Gstr(M')$ is defined as desired.
\end{proof}

This means that we can now construct the moduli spaces of $G$-jet-structures and torsion-free $G$-jet-structures in a nice way:
\begin{definition}
  \label{def:space-of-torsion-free-jet-structures}
  Let $M$ be a formal $\DD_e$-space and $\chi:\BG\to\BAut(\DD_e)$ a pointed map.
  \begin{enumerate}
  \item The \notion{moduli space} of $G$-jet-structures on $M$ is given as
    \[ \sum_{M':\BAut(M)}\Gstr(M').\]
  \item The \notion{moduli space of torsion-free $G$-jet-structures} on $M$ is given as
  \[ \sum_{M':\BAut(M)}\sum_{\Theta:\Gstr(M')}\text{torsion-free}(\Theta). \]
  \end{enumerate}
\end{definition}

\subsection*{Conclusion}

While we did not further discuss this, we expect that the homotopy
type theory developed here has interpretation in suitable
$\infty$-toposes equipped with a fibered idempotent $\infty$-monad.
Our abstract construction of moduli spaces of torsion-free $G$-jet-structures
should then have a translation to a corresponding
construction internal to any of these $\infty$-toposes.
When written out in terms of traditional higher category theory, say
as simplicial sheaves, these objects will look rather
complicated and be cumbersome to work with.
Our abstract language should hence serve to make the development of higher Cartan Geometry
in $\infty$-toposes tractable.

In addition to having little restrictions on models,
the abtract way of working with just one monadic modality is also very clear and very suitable for formalization, since no axioms have to be postulated.
Yet the author does not believe that this line of work should be continued on the level of abstraction used in this article.
In the recent, more concrete framework from \cite{sag}, admitting synthetic treatment of algebraic geometry, calculations were crucial in the advances made.
The appraoch to cohomology developed there is likely to have a differential geometric analogue.
It should be fruitful to find extensions of the common axioms of synthetic differential geometry inspired by that research.
In such an extension it might be possible to show, that the fiber bundles defined in \Cref{def-fiber-bundle} are actually fiber bundles in the usual sense of local triviality with respect to a topology.

The idea for the approach to cohomology in \cite{sag} was to use the ``higher-topos'' approach\footnote{The introduction of \cite{Lurie} explains how higher toposes offer a good perspective on cohomology.}
of just mapping into a higher type, usually an Eilenberg-MacLane space, and analysing the 0-truncation of the resulting function type.
This is quite easy to implement in homotopy type theory and was considered early in the history of the subject\footnote{Written down by Michael Shulman in a \href{https://homotopytypetheory.org/2013/07/24/cohomology/}{blog post} which focuses on the connection with the axiom of choice.}.
One important insight from \cite{sag} is that the notion of local triviality, that comes from the topology of a (higher) sheaf topos,
is internally accessible by a choice principle which is somewhat similar to the WISC axiom from \cite{vandenberg-moerdijk-amc}.
So finding the proposed extension of synthetic differential geometry amounts to checking if there are choice principles in differential geometry,
that can be used to make cohomological calculations.
In the synthetic differential geometry, there already is an axiom called the ``covering principle'', which would be a consequence of any reasonable choice axiom for synthetic differential geometry.\footnote{Pointed out by David Jaz Myers in a discussion after an online talk. In his comment, Myers also implicitly suggests a choice principle base on countable covers for the metric topology.  }

The connection to synthetic algebraic geometry does not stop at cohomology --
there is also synthetic algebro geometric work on formally étale maps \cite{diffgeo-draft} which suggests an extension of the Kock-Lawvere axiom for synthetic differential geometry.
If it can be shown that such an extension is supported by models, the theory in the last of a series of articles by Myers \cite{david-fibrations}, \cite{david-groups} and \cite{david-orbifolds}
could be simplified.

The author believes that the best way to continue the work presented in this article
is to find an extension of the usual axioms of synthetic differential geometry
and use all of the recent advances to compute examples of $V$-manifolds, maps between them, cohomology groups, $G$-jet-structures and their moduli spaces.
We expect that when working internally, computations are a lot more feasible
and should be used to establish correspondences to the classical theory
as well as to guide an expansion of the work in this article into a synthetic higher Cartan Geometry.

\printindex

\printbibliography

\end{document}